\documentclass[12pt]{article}

\usepackage{amsmath,amsthm,amssymb,amscd,easywncy,mathdots,enumerate}
\usepackage{stmaryrd}
\usepackage{setspace}
\usepackage{young}
\usepackage{youngtab}
\usepackage{ytableau}
\usepackage{wrapfig}
\usepackage{shuffle}
\usepackage{pgf}
\usepackage{tikz}
\usepackage{subfig}
\usetikzlibrary{arrows,automata}

\usepackage{float}
\usetikzlibrary{calc}
\usepackage[belowskip=-20pt]{caption}

\allowdisplaybreaks

\theoremstyle{definition}
\newtheorem{theorem}{Theorem}[section]
\newtheorem{proposition}[theorem]{Proposition}

\newtheorem{definition}[theorem]{Definition}
\newtheorem{lemma}[theorem]{Lemma}

\newtheorem{ex}[theorem]{Example}

\numberwithin{equation}{section}

\newcommand{\R}{\mathbb R}
\newcommand{\mz}{\mathcal Z}

\newcommand{\Z}{{\mathbb Z}}
\newcommand{\N}{{\mathbb N}}
\newcommand{\kk}{{\bf k}}
\newcommand{\mm}{{\bf m}}
\newcommand{\ck}{{\bf k^\prime}}
\newcommand{\C}{{\mathbb C}}
\newcommand{\Q}{{\mathbb Q}}

\DeclareMathOperator{\sign}{\operatorname{sign}}
\DeclareMathOperator{\YT}{\operatorname{YT}}
\DeclareMathOperator{\OYT}{\operatorname{OYT}}
\DeclareMathOperator{\SSYT}{\operatorname{SSYT}}
\DeclareMathOperator{\Cor}{\operatorname{Cor}}
\DeclareMathOperator{\F}{\operatorname{F}}

\newcommand{\vsmall}{\rotatebox[origin=c]{-90}{$<$}}

\title{Interpolated Schur multiple zeta values}
\author{Henrik Bachmann\footnote{email : henrik.bachmann@math.nagoya-u.ac.jp, Nagoya University}}
\begin{document}
\date{\today}
\maketitle

\begin{abstract} Inspired by a recent work of M.~Nakasuji, O.~Phuksuwan and Y.~Yamasaki we combine interpolated multiple zeta values and Schur multiple zeta values into one object, which we call interpolated Schur multiple zeta values. Our main result will be a Jacobi-Trudi formula for a certain class of these new objects. This generalizes an analogous result for Schur multiple zeta values and implies algebraic relations between interpolated multiple zeta values. 
\end{abstract}

\section{Introduction}

For $k_1,\dots,k_{r-1} \geq 1, k_r \geq 2$ the multiple zeta value $\zeta(k_1,\ldots,k_r)$ is defined by
\begin{equation} \label{eq:mzv}
\zeta(k_1,\ldots,k_r)=\sum_{0<m_1<\cdots<m_r} \frac{1}{m_1^{k_1}\cdots m_r^{k_r}} \,.
\end{equation}
The $\Q$-vector space spanned by all multiple zeta values will be denoted by $\mz$. Another version of these numbers are given by the multiple zeta-star values, defined by
\begin{equation} \label{eq:mzvstar}
\zeta^\star(k_1,\ldots,k_r)=\sum_{0<m_1\leq\cdots \leq m_r} \frac{1}{m_1^{k_1}\cdots m_r^{k_r}} \,.
\end{equation}
Every $\zeta$ can be written as a linear combination of $\zeta^\star$ and vice versa. In \cite{SY} Yamamoto introduced an interpolated version of these two real numbers, denoted by $\zeta^t(k_1,\ldots,k_r)$, which is a polynomial in $t$ with coefficients in $\mz$. To give its definition we first define for numbers $m_1\leq m_2 \leq \dots \leq m_r$ the number of their equalities by
\[e(m_1,\dots,m_r) = \sharp \left\{1\leq  i \leq r-1 \mid m_i = m_{i+1} \right\} \,. \]
With this the $\zeta^t$ are defined\footnote{In the original paper \cite{SY} a slightly different notation is used, but the $\zeta^t$ there are exactly the same as here.} by
\begin{equation}  \label{eq:mzvt}
\zeta^t(k_1,\ldots,k_r)=\sum_{0<m_1\leq\cdots \leq m_r} \frac{t^{e(m_1,\dots,m_r)}}{m_1^{k_1}\cdots m_r^{k_r}}\,.
\end{equation}
These are polynomials in $\mz[t]$, where the coefficient of $t^d$ is given by the sum of all $\zeta(k_1 \,\square\, k_2 \,\square\, \dots  \,\square\, k_r)$ with exactly $d$ of the $\square$ being a plus "$+$" and the other $\square$ being a comma "$,$". For example in the case $r = 3$ we have
\[\zeta^t(k_1,k_2,k_3) = \zeta(k_1,k_2,k_3)+ \big(  \zeta(k_1+k_2,k_3)+ \zeta(k_1,k_2+k_3) \big)\, t +  \zeta(k_1+k_2+k_3)\, t^2\,.\]

 By definition these interpolated versions satisfy $\zeta^0 = \zeta$ and $\zeta^1 = \zeta^\star$. Another generalization of multiple zeta and zeta-star values are given by Schur multiple zeta values, which were first introduced in \cite{NPY}.  There the authors proved Jacobi-trudi formulas for these Schur multiple zeta values and the  present work was inspired by this result.   
 Instead of introducing a parameter $t$, Schur multiple zeta values replace an index set $(k_1,\dots,k_r)$ by a Young tableau $\kk = (\lambda, (k_{i,j}))$. Here $\lambda$ is a partition of a natural number, i.e. a non-decreasing sequence  $(\lambda_1,\dots,\lambda_h)$ of non-negative integers and the $k_{i,j}$ are arbitrary integers for indices $(i,j)$ with $1 \leq i \leq h$ and $1 \leq j \leq \lambda_i$. We will make this more precise in the next Section. 
 The tuple $\kk$ can be represented by a Young tableau and we denote its conjugate by $\kk^\prime = (\lambda^\prime, (k_{j,i}))$. 
 For example for $\lambda=(3,3,2,1)$ it is $\lambda^\prime=(4,3,2)$ and we write
\[\kk = {\footnotesize \ytableausetup{centertableaux, boxsize=1.8em}
\begin{ytableau}
k_{1,1} & k_{1,2} & k_{1,3} \\
k_{2,1} & k_{2,2} & k_{2,3} \\
k_{3,1} & k_{3,2} \\
k_{4,1}
\end{ytableau} \qquad \text{and} \qquad \kk^\prime =\begin{ytableau}
k_{1,1} & k_{2,1}& k_{3,1}& k_{4,1}    \\
k_{1,2} & k_{2,2}  & k_{3,2}\\
k_{1,3} &  k_{2,3}
\end{ytableau} \, }\,.\]
In \cite{NPY} the authors\footnote{In \cite{NPY} the Schur multiple zeta function are denoted by $\zeta_\lambda({\bf s})$. We omit the subscript $\lambda$ here since we will consider their truncated version later, which will be denoted $\zeta_N$.} defines $\zeta(\kk) \in \R$ for such a Young tableau $\kk$. For example for $\lambda = (2,1)$ and numbers $a\geq 1,b,c \geq 2$ it is
\[\zeta\left({\footnotesize \ytableausetup{centertableaux, boxsize=1.1em}
	\begin{ytableau}
	a & b  \\
	c
	\end{ytableau}}\right) = \sum_{
	\arraycolsep=1.4pt\def\arraystretch{0.8}
	\begin{array}{ccc}
	0<&m_a &\leq m_b \\
   &\vsmall & \, \\
	&m_c & \, 
	\end{array} } \frac{1}{m_a^a\cdot  m_b^b \cdot m_c^c}  \,.\]
Schur multiple zeta values generalize multiple zeta and zeta-star values in the sense that we recover these in the special cases $\lambda=(1,\dots,1)$ and  $\lambda=(r)$, i.e. we have 
\[ \zeta(k_1,\dots,k_r) = \zeta\left({\footnotesize \ytableausetup{centertableaux, boxsize=1.5em}
	\begin{ytableau}
	k_1  \\
	\vdots \\
	k_r
	\end{ytableau}}\right) \quad \text{ and }\qquad \zeta^\star(k_1,\dots,k_r) =\zeta\left({\footnotesize \ytableausetup{centertableaux, boxsize=1.5em}
	\begin{ytableau}
	k_1 & \dots &k_r
	\end{ytableau}}\right) \,.   \]

In this note we combine the interpolated multiple zeta values and the Schur multiple zeta values into one object $\zeta^t(\kk) \in \mz[t]$, such that 

\[ \zeta^0(\kk) = \zeta(\kk) \quad \text{and}  \quad \zeta^t\left({\footnotesize \ytableausetup{centertableaux, boxsize=1.5em}
	\begin{ytableau}
	k_1  \\
	\vdots \\
	k_r
	\end{ytableau}} \right) = \zeta^t(k_1,\dots,k_r) \,.   \]

Moreover the transformation $t \rightarrow 1-t$ will correspond to conjugating the Young tableau, i.e. for example we will have 

\[\zeta^t\left({\footnotesize \ytableausetup{centertableaux, boxsize=1em}
	\begin{ytableau}
	a & b   \\
	c \\
	d
	\end{ytableau}} \right)  = \zeta^{1-t}\left({\footnotesize \ytableausetup{centertableaux, boxsize=1em}
	\begin{ytableau}
	a & c & d  \\
	b 
	\end{ytableau}} \right) \,.\]
 The main result of this work is the following Theorem on Jacobi-Trudi like formulas for interpolated Schur multiple zeta values:

\begin{theorem}\label{thm:thm1}
For a sequence  $(a_i)_{i\in \Z}$ of integers $a_i \geq 2$ and a Young tableau given by $\kk=(\lambda, (k_{i,j}))$ with $k_{i,j}=a_{j-i}$, we have the following identity
	\[ \zeta^t(\kk) = \det(\, \zeta^t(a_{j-1}, a_{j-2}, \dots , a_{j-(\lambda^\prime_i+j-i)}) \,)_{1\leq i,j \leq \lambda_1}\,,\]
	where we set $\zeta^t(a_{j-1}, a_{j-2}, \dots , a_{j-(\lambda^\prime_i+j-i)})$ to be $\begin{cases}
1 \, \text{  if  }  \lambda^\prime_i-i +j =0  \\
0 \, \text{  if  }  \lambda^\prime_i-i +j <0 
\end{cases}$.
\end{theorem}
This Theorem is a general version of Theorem 1.1 in \cite{NPY}, where the cases $t=0$ and $t=1$ are proven. As an application we obtain algebraic relations between interpolated multiple zeta values:

\begin{theorem} \label{thm:thm2}Given an arbitrary sequence of integers $a_j \geq 2$, $i\in \Z$ and a partition $\lambda=(\lambda_1,\dots,\lambda_h)$ with conjugate  $\lambda^\prime=(\lambda^\prime_1,\dots,\lambda^\prime_{\lambda_1})$ the following identity holds
	\[ \det(\, \zeta^{1-t}(a_{1-j}, a_{2-j}, \dots , a_{(\lambda_i+j-i)-j}) \,)_{1\leq i,j \leq h}= \det(\, \zeta^t(a_{j-1}, a_{j-2}, \dots , a_{j-(\lambda^\prime_i+j-i)}) \,)_{1\leq i,j \leq \lambda_1}\,,\]
	where we use the same definition of  $\zeta^t(a_m,\dots,a_n)$ for $n<m$ as before. 
\end{theorem}

\begin{ex}
	\begin{enumerate}[i)]
		\item Choosing $\lambda=(r)$ and setting $a_j = k_{j+1}$ we obtain for $k_1,\dots,k_r \geq 2$ the identity
		\begin{equation}{\small \zeta^{1-t}(k_1,\dots,k_r) = \det \begin{pmatrix}
		\zeta^t(k_1) & \zeta^t(k_2,k_1) & \zeta^t(k_3,k_2,k_1) &\dotsb & \zeta^t(k_r,\dots,k_1)\\
		1 			 & \zeta^t(k_2) 		& \zeta^t(k_3,k_2) & &\zeta^t(k_r,\dots,k_2) \\
		0 &	1& \zeta^t(k_3) &\zeta^t(k_4,k_3)& \vdots \\
		\vdots		&\ddots	& \ddots & \ddots&\zeta^t(k_r,k_{r-1})\\
		0		&\dotsb	&0 &1&\zeta^t(k_r)
		\end{pmatrix} \,.}
		\end{equation}
		\item Setting $\lambda=(\overbrace{r,\dots,r}^{r})$ and $a_j = k_{|j|+1}$ for $k_1,\dots,k_r\geq 2$  Theorem \ref{thm:thm1} implies that the Polynomial $ M_{k_1,\dots,k_r}(t) \in \mz[t]$ defined by
\[{\small M_{k_1,\dots,k_r}(t) = \det \begin{pmatrix}
\zeta^t(k_1,\dots,k_r) 			& \zeta^t(k_2,k_1,\dots,k_r) & \dotsb 			& \zeta^t(k_r,\dots,k_1,\dots,k_r)\\
\zeta^t(k_1,\dots,k_{r-1})		&  	\ddots	 				 & 	\ddots & \vdots 	\\
\vdots 							&			\ddots					 &  	\ddots	& \zeta^t(k_{r},\dots,k_1,k_2) \\
\zeta^t(k_1)	 				&	\zeta^t(k_2,k_1)		 &  \dotsb 			& \zeta^t(k_{r},\dots,k_1)
\end{pmatrix} \,.}\]		
		satisfies
			\begin{equation} M_{k_1,\dots,k_r}(t)= M_{k_1,\dots,k_r}(1-t)\,.	\end{equation}
	\end{enumerate}
\end{ex}

We will prove both Theorems for truncated version of interpolated (Schur) multiple zeta values. First we will recall the definition of Schur mutliple zeta values and then generalize this to their interpolated version. After this we will recall the Lemma of  Lindstr\"om, Gessel and Viennot which will be needed for the proof of Theorem \ref{thm:thm1} and its more general case for the truncated versions. In the last Section we will discuss a generalization of our results.\\

{\bf Acknowledgments}\\
The author would like to thank M. Nakasuji for introducing him to this topic and S. Kadota, N. Matthes,  Y. Suzuki and Y. Yamasaki for corrections on this note and fruitful discussions on Schur multiple zeta values. This project started while the author was an Overseas researcher under Postdoctoral Fellowship of Japan Society for the Promotion of Science. It was partially supported by JSPS KAKENHI Grant Number JP16F16021. Finally, the author would like to thank the Max-Planck Institute for Mathematics in Bonn for hospitality and support. 

\section{Interpolated Schur multiple zeta values}
In this section we will introduce interpolated Schur multiple zeta values. Since we will treat their truncated version in most of the cases and especially in the Proof of the main theorem, we recall that for a natural number $N\geq1$ the truncated version of \eqref{eq:mzv}, \eqref{eq:mzvstar} and \eqref{eq:mzvt} are given by
\begin{align*} 
\zeta_N(k_1,\ldots,k_r)&=\sum_{0<m_1<\cdots<m_r<N} \frac{1}{m_1^{k_1}\cdots m_r^{k_r}} \in \Q \,, \\ \zeta^\star_N(k_1,\ldots,k_r)&=\sum_{0<m_1\leq \cdots\leq m_r<N} \frac{1}{m_1^{k_1}\cdots m_r^{k_r}} \in \Q
\end{align*}
and \begin{equation*} \label{eq1_2}
\zeta^t_N(k_1,\ldots,k_r)=\sum_{0<m_1\leq\cdots \leq m_r<N} \frac{t^{e(m_1,\dots,m_r)}}{m_1^{k_1}\cdots m_r^{k_r}} \in \Q[t].
\end{equation*}
Before we define the truncated interpolated Schur multiple zeta values we will need to introduce some notation.
\begin{definition}
\begin{enumerate}[i)]
	\item By a partition of a natural numbers $n$ we denote a tuple $\lambda = (\lambda_1,\dots,\lambda_h)$, where $\lambda_1 \geq \dots \geq \lambda_h \geq 1$ and $n = |\lambda|:= \lambda_1 + \dots + \lambda_h$. Its conjugate is denoted by  $\lambda^\prime=(\lambda^\prime_1,\dots,\lambda^\prime_{\lambda_1})$ and it is defined by transposing the corresponding Young diagram. In the case $\lambda=(5,2,1)$ it is $\lambda^\prime=(3,2,1,1,1)$ which can be visualized by  
	\[{\footnotesize \ytableausetup{centertableaux, boxsize=0.5em}
		\lambda = \begin{ytableau}
		\,& \,& \,& \,&\,\\
		\,& \,\\
		\,
		\end{ytableau}}  \,\longrightarrow \,\lambda^\prime = {\footnotesize \begin{ytableau}
		\,& \,& \,\\
		\,& \,\\
		\,\\
		\,\\
		\,
		\end{ytableau}} \,. \]
	\item Using similar notation as in \cite{NPY} we define for a partition the set $\lambda$ \[D(\lambda) = \left\{(i,j) \in \Z^2 \mid 1 \leq i \leq h\,, 1 \leq j \leq \lambda_i \right\}\]which describe the coordinates of the corresponding Young diagram. 
	\item By a Young tableau of shape $\lambda$ we will denote a tupel $\kk = (\lambda, (k_{i,j}))$, where $k_{i,j} \in \Z$ are integers for indices $(i,j) \in D(\lambda)$. By $\YT(\lambda)$ we denote the set of all Young tableau of shape $\lambda$. 
\end{enumerate}
\end{definition}

We will use Young tableau for the replacement of the index set as well as a replacement for the summation domain. Instead summing over ordered pairs of natural numbers we will sum over "ordered" Young tableau.

\begin{definition}
	\begin{enumerate}[i)]
		\item  For a partition $\lambda = (\lambda_1,\dots,\lambda_h)$ and a natural numbers $N \geq 1$ we define the following set of ordered Young tableaux 
		\begin{align*}
		\OYT_N(\lambda) = \left\{ \mm  = (\lambda, (m_{i,j})) \in \YT(\lambda) \bigm\vert \begin{array}{c}
		0 < m_{i,j} < N\\
		m_{i,j} \leq m_{i+1,j} \text{ and } m_{i,j} \leq m_{i,j+1} \\
		m_{i,j} < m_{i+1,j+1}
		\end{array}   \right\}\,.
		\end{align*}
		The entries of these Young tableaux are ordered from the top left to the bottom right and there are no equalities on the diagonals. By 	$\OYT(\lambda) = \bigcup_{N\geq 1}	\OYT_N(\lambda)$ we denote the set of all ordered Young tableaux.  
		\item For an ordered Young tableau $ \mm  = (\lambda, (m_{i,j})) \in 	\OYT(\lambda) $ we define the number of vertical equalities by 
		\[v(\mm) =  \#\left\{ (i,j) \in D(\lambda) \mid m_{i,j} = m_{i+1,j} \right\}  \]
		and the number of horizontal equalities by 
		\[h(\mm) =  \#\left\{ (i,j) \in D(\lambda) \mid m_{i,j} = m_{i,j+1} \right\}  \,.\]
	\end{enumerate}
\end{definition}

For example in the case $N=4$ and $\lambda=(2,2)$ we have the following set
\[\OYT_4\left(\,{\footnotesize \ytableausetup{centertableaux, boxsize=1em}
	\begin{ytableau}
	\,& \, \\
	\,& \, 
	\end{ytableau}}\,\right) =\left\{ {\footnotesize 	\begin{ytableau}
	1& 1 \\
	1& 2
	\end{ytableau}} \,,\,{\footnotesize 	\begin{ytableau}
	1& 2 \\
	1& 2
	\end{ytableau}}\,,\,{\footnotesize 	\begin{ytableau}
	1& 1 \\
	2& 2
	\end{ytableau}}\,,\,{\footnotesize 	\begin{ytableau}
	1& 2 \\
	2& 2
	\end{ytableau}}\,,\,{\footnotesize 	\begin{ytableau}
	1& 1 \\
	1& 3
	\end{ytableau}} \,,\,{\footnotesize 	\begin{ytableau}
	1& 2 \\
	1& 3
	\end{ytableau}}\,,\dots,\,{\footnotesize 	\begin{ytableau}
	2& 3 \\
	3& 3
	\end{ytableau}}\right\} \,.\]
As an example for the number of vertical and horizontal equalities we give
\[ v\left({\footnotesize 	\begin{ytableau}
	2& 2 & 3 \\
	2& 3 \\
	2& 4 \\
	2
	\end{ytableau}}\right) = 3\quad \text{and} \quad  h\left({\footnotesize 	\begin{ytableau}
	2& 2 & 3 \\
	2& 3 \\
	2& 4 \\
	2
	\end{ytableau}}\right) = 1\,.  \]

\begin{definition} For a Young tableau $\kk = (\lambda, (k_{i,j})) \in \YT(\lambda)$ and a natural number $N\geq 1$ we define the truncated interpolated Schur multiple zeta value by
\begin{equation}\label{def:interpolsmzv}
 \zeta_N^t(\kk) = \sum_{\substack{\mm \in \OYT_N(\lambda)\\\mm = (\lambda, (m_{i,j}))}} t^{v(\mm)}  (1-t)^{h(\mm)} \prod_{(i,j) \in D(\lambda)} \frac{1}{m_{i,j}^{k_{i,j}}} \in \Q[t]\,. 
\end{equation}

\end{definition}
We will first explain why these polynomials can be seen as a generalization of the Schur multiple zeta values defined in \cite{NPY}.
For this we recall that the semi-standard Young tableaux $\SSYT(\lambda)$ are ordered Young tableaux without any vertical equalities. In other words they can be defined by $$\SSYT(\lambda) =  \left\{\mm \in 	\OYT(\lambda)  \mid  v(\mm) = 0 \right\}\,.$$ 
Writing  $\SSYT_N(\lambda) = \SSYT(\lambda) \cap 	\OYT_N(\lambda)$ we define the truncated Schur multiple zeta values by
\[ \zeta_N(\kk) = \sum_{\substack{\mm \in \SSYT_N(\lambda)\\\mm = (\lambda, (m_{i,j}))}} \prod_{(i,j) \in D(\lambda)} \frac{1}{m_{i,j}^{k_{i,j}}} \in \Q\,. \]
It is easy to see that we have $\zeta_N^0(\kk) = \zeta_N(\kk)$, since just the terms with $v(\mm) = 0$ contribute to the sum in \eqref{def:interpolsmzv} for the $t=0$ case. In \cite{NPY} the authors define Schur multiple zeta values by taking the limit $N \rightarrow \infty$ of $\zeta_N(\kk)$. The question for which Young tableaux $\kk$ this limit for $\zeta^t_N(\kk)$ exists will be discussed now. For this define the set of corners of a partition $\lambda$ by 
\[  \Cor(\lambda) =  \left\{ (i,j) \in D(\lambda) \mid  (i,j+1) \not\in D(\lambda) \text{ and } (i+1,j) \not\in D(\lambda)    \right\} \subset D(\lambda) \,. \]
For example in the case $\lambda=(3,2,2,1)$ the corners are $\Cor(\lambda) = \{(1,3), (3,2),(4,1)\}$, which we visualize as $\bullet$ in the corresponding Young diagram:
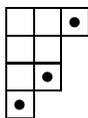
\[{\footnotesize 	\ytableausetup{centertableaux, boxsize=1em}\begin{ytableau}
	\,& \, & \bullet \\
	\,& \, \\
	\,&  \bullet \\
	 \bullet
	\end{ytableau}}\]
With this definition we can state the convergence region for interpolated Schur multiple zeta values.
\begin{lemma}
 In the case $k_{i,j} \geq 2$ for $(i,j) \in \Cor(\lambda)$ and $k_{i,j}\geq 1$ otherwise, the limit
 \[ \zeta^t(\kk) := \lim\limits_{N \rightarrow \infty}  \zeta_N^t(\kk)   \]
 exists and it is $\zeta^t(\kk) \in \mz[t]$.
 We call the polynomials $\zeta^t(\kk)$ interpolated Schur multiple zeta values.
\end{lemma} 
\begin{proof} This follows from the analogue statement for Schur multiple zeta values (Lemma 2.1) in \cite{NPY}. 
 The fact that these are elements in $\mz[t]$ can be proven by easy combinatorial arguments which we will omit here.  
\end{proof}

\begin{ex}
In the case $\lambda=(2,1)$ one can check that for $a\geq 1, b,c \geq 2$
\begin{align*}	\zeta^t \left(\begin{ytableau}
	a & b  \\
	c
	\end{ytableau}\right) &= \zeta(a,b,c) +\zeta(a,c,b) + \zeta(a+b,c)+\zeta(a,b+c) \\
	&+ \big( \zeta(a+c,b) - \zeta(a+b,c) + \zeta(a+b+c) \big) \cdot t - \zeta(a+b+c) \cdot  t^2  \,.
\end{align*}
\end{ex}
In the special case $\lambda = (1,\dots,1)$ we obtain the following sum
\[  \zeta_N^t\left({\footnotesize \ytableausetup{centertableaux, boxsize=1.5em}
	\begin{ytableau}
	k_1  \\
		\scriptstyle  \vdots \\
	k_r
	\end{ytableau}} \right) = \sum_{0 < m_1 \leq\dots \leq m_r < N} \frac{t^{e(m_1,\dots,m_r)}}{m_1^{k_1}\cdots m_r^{k_r}} = \zeta^t_N(k_1,\dots,k_r)\,,   \]
	where we recall that $e(m_1,\dots,m_r)$ are the number of equalities between the $m_j$, which in this case correspond to the number of vertical equalities $v(\mm)$.
Therefore in the case $k_1,\dots,k_{r-1}\geq 1$ and $k_r\geq 2$ we can take the limit $N \rightarrow \infty$ and obtain the interpolated multiple zeta value $\zeta^t(k_1,\dots,k_r)$.

Even though we are more interested in the interpolated Schur multiple zeta values, we will work with the truncated versions in the following. 
It should also be remarked that the weights $k_{i,j}$ can be arbitrary (in particular negative) integers in the truncated case. 
The first property of truncated interpolated Schur multiple zeta values is the behavior under the transformation $t\rightarrow 1-t$, which will be the key fact for the proof of Theorem \ref{thm:thm2} in the introduction. 

\begin{proposition}\label{prop:t1mt}
The transformation $t\rightarrow 1-t$ corresponds to conjugation. For every Young tableau $\kk \in \YT(\lambda)$ and every $N\geq 1$ we have
\[\zeta_N^t(\kk) = \zeta_N^{1-t}(\ck) \,.\]
\end{proposition}
\begin{proof}
	This is clear by definition of interpolated Schur multiple zeta values. The number of vertical and horizontal equalities switch when replacing $\kk$ by $\kk^\prime$. Therefore by substituting $t$ with $1-t$ the total sum does not change. 
\end{proof}
\section{The Lemma of Lindstr\"om, Gessel and Viennot}
In this section we will recall some notations from \cite{AZ}  to state the Lemma of Lindstr\"om, Gessel and Viennot. This will be the key ingredient for the Proof of Theorem \ref{thm:thm1} and we will give an example for the special case $t=1$. 

 Suppose we are given an arbitrary commutative ring $\mathcal{R}$  (which will be $\Q$ and $\Q[t]$ in our cases) and a finite acyclic directed Graph $G=(V,E,w)$. Here by $V$ we denote the set of vertices, by $E$ we denote the set of directed edges and by the map $w: E \rightarrow \mathcal{R}$ we denote its weight function. 
 For a directed path $P$ from a vertex $A \in V$ to a vertex $B \in V$, written shortly $P: A \rightarrow B$, we define its weight by
\[ w(P) := \prod_{e \in P} w(e)\,,\]
where the product runs over all edges on the path $P$.   We further define the sum of all path weights from an $A$ to $B$ by
\[ w(A,B) := \sum_{P: A \rightarrow B} w(P)  \,. \]
For two sets of vertices $\mathcal{A}=\{A_1,\dots,A_n\} \subset V$ and $\mathcal{B}=\{B_1,\dots,B_n\} \subset V$  we define a vertex-disjoint path system $\mathcal{P}$ from $\mathcal{A}$ to $\mathcal{B}$, denoted by $\mathcal{P} :\mathcal{A} \rightarrow \mathcal{B}$, as a permutation $\sigma \in \Sigma_n$ together with $n$ pairwise vertex-disjoint paths $P_i : A_i \rightarrow B_{\sigma{i}}$. We write $\sign \mathcal{P} = \sign \sigma$ and define
\[ w(\mathcal{P}) := \prod_{j=1}^n w(P_i) \,.\]
\begin{lemma}(Lindstr\"om, Gessel, Viennot) \label{lem:LGV}Let $G=(V,E,w)$ be a finite weighted acyclic directed graph, $A=\{A_1,\dots,A_n\}$ and $B=\{B_1,\dots,B_n\}$ two $n$-sets of vertices. Then 
	\[  \sum_{\mathcal{P} \,:\,\mathcal{A} \rightarrow \mathcal{B}} \sign\mathcal{P} \cdot w(\mathcal{P})= \det( w(A_i,B_j) )_{1\leq i,j \leq n}  \]
\end{lemma}
\begin{proof}
This is a classical result which can be found in this form in \cite{AZ} Chapter 25.
\end{proof}

Our goal will be to construct the correct directed Graph and weights, such that the right-(resp. left-)hand side of Lemma \ref{lem:LGV} will give us the right-(resp. left-)hand side of Theorem \ref{thm:thm1}. We first illustrate with an example how to obtain multiple zeta-star values by summing over paths of certain Graphs before we consider the interpolated versions afterwards. 

\begin{ex}\label{ex:LGVExample}
 i) Consider the Lattice Graph $G=(V,E,w)$ with vertices $$V=\{(x,y) \mid 1 \leq x,y \leq 6\}\,.$$ From each vertex we construct a directed edge to the vertex on the right and to the vertex below, i.e. paths can just run from the top left to the bottom right. We set $A_1 = (1,6)$ to be the vertex on the top left and $B_1 = (6,1)$ to be the vertex on the bottom right.
 For a horizontal edge $e$ from $(x,y)$ to $(x+1,y)$ we define the weight by $w(e) = y^{-a_{x-5}}$, with arbitrary integers $(a_j)_{j\in\Z}$. For a vertical edge $e$ we set $w(e)=1$.
  One possible path from $A_1$ to $B_1$ together with its weight is shown in Figure \ref{fig:exampleL6}. 
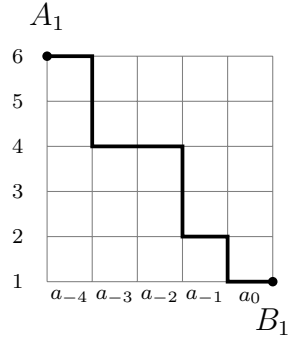
\begin{figure}[H]
	\centering
	\begin{tikzpicture}[scale=0.3]
\draw[step=2, thin, gray] (0,0) grid (10,10); 

\draw [line width=0.5mm] (0,10) -- (2,10) -- (2,6) -- (4,6) -- (6,6) -- (6,2) -- (8,2)--(8,0) -- (10,0) ; 

\fill[black] (10,0) circle (6pt);
\fill[black] (0,10) circle (6pt);
\coordinate [label=above:$A_1$] (R) at (0,10.7);
\coordinate [label=below:$B_1$] (R) at (10,-0.7);

\node (c) at (-1.3,0) {\scriptsize $1$};
\node (c) at (-1.3,2) {\scriptsize $2$};
\node (c) at (-1.3,4) {\scriptsize $3$};
\node (c) at (-1.3,6) {\scriptsize $4$};
\node (c) at (-1.3,8) {\scriptsize $5$};
\node (c) at (-1.3,10) {\scriptsize $6$};


\node (c) at (1,-0.7) {\scriptsize $a_{-4}$};
\node (c) at (3,-0.7) {\scriptsize $a_{-3}$};
\node (c) at (5,-0.7) {\scriptsize $a_{-2}$};
\node (c) at (7,-0.7) {\scriptsize $a_{-1}$};
\node (c) at (9,-0.7) {\scriptsize $a_0$};
\end{tikzpicture}\vspace*{-7mm}
	\caption{A path $P$ from $A_1$ to $B_1$ with weight $w(P)=1^{-a_{0}} 2^{-a_{-1}} 4^{-a_{-2}} 4^{-a_{-3}} 6^{-a_{-4}}.$}\label{fig:exampleL6}
\end{figure}
\vspace*{3mm}
Summing over all possible path from $A_1$ to $B_1$ we see that we get the truncated multiple zeta-star value (for $N=7$ in this case):
  \[ w(A_1,B_1) := \sum_{P: A_1 \rightarrow B_1} w(P) = \sum_{0 < m_1 \leq \dots \leq m_5 < 7} \frac{1}{m_1^{a_0} \dots m_5^{a_{-4}}} = \zeta^\star_7(a_0,\dots,a_{-4})  \,. \]  

ii) Now consider the similar Graph in Figure \ref{fig:exampleL62}, where the weights are defined in the same way as before but where we now consider four Points from the Set $\mathcal{A}=\{A_1,A_2\}$ and $\mathcal{B}=\{B_1,B_2\}$. 

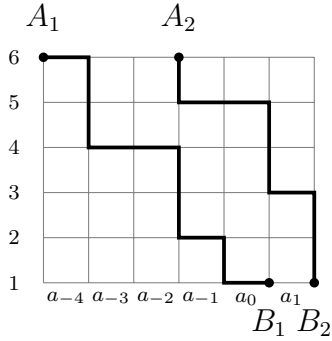
\begin{figure}[H]
	\centering
	\begin{tikzpicture}[scale=0.3]
\draw[step=2, thin, gray] (0,0) grid (12,10); 

\draw [line width=0.5mm] (0,10) -- (2,10) -- (2,6) -- (4,6) -- (6,6) -- (6,2) -- (8,2)--(8,0) -- (10,0) ; 

\draw [line width=0.5mm] (6,10) -- (6,8) -- (8,8) -- (10,8) -- (10,6) -- (10,4) -- (12,4) -- (12,2)-- (12,0) ;

\fill[black] (10,0) circle (6pt);
\fill[black] (0,10) circle (6pt);
\fill[black] (12,0) circle (6pt);
\fill[black] (6,10) circle (6pt);
\coordinate [label=above:$A_1$] (R) at (0,10.7);
\coordinate [label=below:$B_1$] (R) at (10,-0.7);
\coordinate [label=above:$A_2$] (R) at (6,10.7);
\coordinate [label=below:$B_2$] (R) at (12,-0.7);

\node (c) at (-1.3,0) {\scriptsize $1$};
\node (c) at (-1.3,2) {\scriptsize $2$};
\node (c) at (-1.3,4) {\scriptsize $3$};
\node (c) at (-1.3,6) {\scriptsize $4$};
\node (c) at (-1.3,8) {\scriptsize $5$};
\node (c) at (-1.3,10) {\scriptsize $6$};

\node (c) at (1,-0.7) {\scriptsize $a_{-4}$};
\node (c) at (3,-0.7) {\scriptsize $a_{-3}$};
\node (c) at (5,-0.7) {\scriptsize $a_{-2}$};
\node (c) at (7,-0.7) {\scriptsize $a_{-1}$};
\node (c) at (9,-0.7) {\scriptsize $a_0$};
\node (c) at (11,-0.7) {\scriptsize $a_1$};
\end{tikzpicture}\vspace*{-7mm}
	\caption{Two paths with weights $1^{-a_{0}} 2^{-a_{-1}} 4^{-a_{-2}} 4^{-a_{-3}} 6^{-a_{-4}}$ and $5^{-a_{-1}} 5^{-a_{0}} 3^{-a_{1}}$}\label{fig:exampleL62}
\end{figure}
\vspace*{3mm}
Clearly we have with the same arguments as before
\[\det( w(A_i,B_j) )_{1\leq i,j \leq 2} = \begin{pmatrix}
w(A_1,B_1) & w(A_1,B_2)\\
w(A_2,B_1) & w(A_2,B_2)
\end{pmatrix} =  \begin{pmatrix}
\zeta^\star_7(a_0,\dots,a_{-4})  & \zeta^\star_7(a_1,a_0,\dots,a_{-4}) \\
\zeta^\star_7(a_{-1},a_{0})  & \zeta^\star_7(a_{-1},a_{0},a_1) 
\end{pmatrix} \,.\]
Due to Lemma \ref{lem:LGV} this equals $\sum_{\mathcal{P} \,:\,\mathcal{A} \rightarrow \mathcal{B}} \sign\mathcal{P} \cdot w(\mathcal{P})$. Since we just consider pairwise vertex-disjoints paths $P_i : A_i \rightarrow B_{\sigma{i}}$ the only choice are the  vertex-disjoint path systems $\mathcal{P}$ with $\sigma = id$ (i.e. $\sign\mathcal{P}=1$) because every path from $A_1$ to $B_2$ would cross a path from $A_2$ to $B_1$. We claim that we then obtain 
\[\sum_{\mathcal{P} \,:\,\mathcal{A} \rightarrow \mathcal{B}} w(\mathcal{P})  = \zeta_7^1\left(\kk \right) \quad \text{ with } \quad \kk = {\footnotesize 	\ytableausetup{centertableaux, boxsize=1.8em}\begin{ytableau}
	a_0& a_1 \\
	a_{-1}& a_0 \\
	a_{-2}& a_{-1} \\
	a_{-3}\\
	a_{-4}
	\end{ytableau}} \,\,.\]
In other words the sum over all  vertex-disjoint path systems in the graph of Figure \ref{fig:exampleL62} gives the truncated interpolated Schur multiple zeta values $\zeta^t_N(\kk)$ for $t=1$ and $N=7$, which we will explain now. 
The part of the sum corresponding to the first column of the above Young tableaux $\kk$ corresponds to the paths $P_1: A_1 \rightarrow B_1$ and the second column to the paths $P_2: A_2 \rightarrow B_2$. 
Setting $t=1$ in the definition of interpolated Schur multiple zeta values means, that we just sum over Young tableaux $\mm$ with no horizontal equalities.  It is easy to check that this exactly corresponds to the fact, that the two paths $P_1$ and $P_2$ in a vertex-disjoint path system $\mathcal{P}$ have disjoint vertices. The  vertex-disjoint path system in Figure \ref{fig:exampleL62} corresponds to the Young tableaux 
\[\mm = {\footnotesize 	\ytableausetup{centertableaux, boxsize=1.3em}\begin{ytableau}
	1& 3 \\
	2& 5 \\
	4& 5 \\
	4\\
	6
	\end{ytableau}} \in \OYT_7\left({\footnotesize 	\ytableausetup{centertableaux, boxsize=0.4em}\begin{ytableau}
	\,& 	\, \\
		\,& 	\, \\
		\,& 	\, \\
		\,\\
		\,
	\end{ytableau}}\right)
\]
which is a summand in the definition of  $\zeta^1_7(\kk)$. Since $\zeta^1 = \zeta^\star$ we get in total 
\[  \zeta_7^1\left(\kk \right) =   \begin{pmatrix}
\zeta^1_7(a_0,a_{-1},a_{-2},a_{-3},a_{-4})  & \zeta^1_7(a_1,a_0,a_{-1},a_{-2},a_{-3},a_{-4}) \\
\zeta^1_7(a_{0},a_{1})  & \zeta^1_7(a_{1},a_{0},a_{-1}) 
\end{pmatrix}  \]
which is the truncated version of Theorem \ref{thm:thm1} in the $t=1$ case. 
\end{ex}

To prove the statement of Theorem \ref{thm:thm1} for general $t$ we will consider a more complicated Graph in the next section.

\section{The $t$-Lattice Graph}
To include the parameter $t$ into the Graph we will construct a Lattice Graph with two vertices on each position $(x,y)$ together with five different types of edges.

\begin{definition}\label{def:tlattice}
We define for a fixed $N \geq 1$ the weighted directed Graph $G^t_N = (V_N,E_N, w)$, which we call \emph{$t$-Lattice Graph}, by the following data:
\begin{enumerate}[i)]
	\item The vertices are given by $V_N = \left\{ v^\ast_{x,y}  \mid x \in \Z, \, 0 \leq y < N ,\, \ast \in \{\bullet, \circ\} \right\}$.
	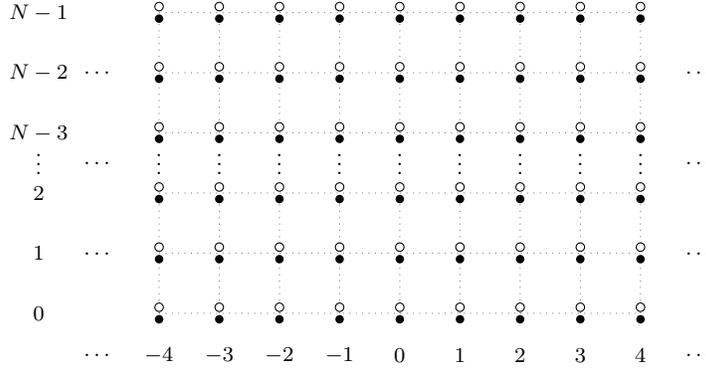
\begin{figure}[H]
		\centering
		\begin{tikzpicture}[scale=0.8]
\newcommand{\dpt}[2]{	
	\fill[white] (#1,#2+0.1) circle (2pt);
	\draw (#1,#2+0.1) circle (2pt);
	\fill[black] (#1,#2-0.1) circle (2pt);
}
\draw[step=1, thin, gray,dotted] (-4,0) grid (4,2); 
\draw[step=1, thin, gray,dotted] (-4,3) grid (4,5);

\node (c) at (-5,-0.7) {\scriptsize $\dots$};
\node (c) at (-4,-0.7) {\scriptsize $-4$};
\node (c) at (-3,-0.7) {\scriptsize $-3$};
\node (c) at (-2,-0.7) {\scriptsize $-2$};
\node (c) at (-1,-0.7) {\scriptsize $-1$};
\node (c) at (0,-0.7) {\scriptsize $0$};
\node (c) at (1,-0.7) {\scriptsize $1$};
\node (c) at (2,-0.7) {\scriptsize $2$};
\node (c) at (3,-0.7) {\scriptsize $3$};
\node (c) at (4,-0.7) {\scriptsize $4$};
\node (c) at (5,-0.7) {\scriptsize $\dots$};
\node (c) at (5,1) {\scriptsize $\dots$};
\node (c) at (5,2.5) {\scriptsize $\dots$};
\node (c) at (5,4) {\scriptsize $\dots$};
\node (c) at (-5,1) {\scriptsize $\dots$};
\node (c) at (-5,2.5) {\scriptsize $\dots$};
\node (c) at (-5,4) {\scriptsize $\dots$};
\node (c) at (-6,1) {\scriptsize $1$};
\node (c) at (-6,0) {\scriptsize $0$};
\node (c) at (-6,2) {\scriptsize $2$};
\node (c) at (-6,2.6) {\scriptsize $\vdots$};
\node (c) at (-4,2.6) {\scriptsize $\vdots$};
\node (c) at (-3,2.6) {\scriptsize $\vdots$};
\node (c) at (-2,2.6) {\scriptsize $\vdots$};
\node (c) at (-1,2.6) {\scriptsize $\vdots$};
\node (c) at (-0,2.6) {\scriptsize $\vdots$};
\node (c) at (1,2.6) {\scriptsize $\vdots$};
\node (c) at (2,2.6) {\scriptsize $\vdots$};
\node (c) at (3,2.6) {\scriptsize $\vdots$};
\node (c) at (4,2.6) {\scriptsize $\vdots$};
\node (c) at (-6,3) {\scriptsize $N-3$};
\node (c) at (-6,4) {\scriptsize $N-2$};
\node (c) at (-6,5) {\scriptsize $N-1$};

\foreach \x in {-4,...,4} {
\dpt{\x}{0}	
\dpt{\x}{1}	
\dpt{\x}{2}
\dpt{\x}{3}	
\dpt{\x}{4}
\dpt{\x}{5}			
	}

\end{tikzpicture}\vspace*{-7mm}
		\caption{Vertices $V_N$ of the $t$-Lattice Graph.}
	\end{figure}
	\vspace*{3mm}
	\item For each position $(x,y)$ (with $y>0$) we will have five different types of outgoing edges 
	\[E_N := \left\{ e^i_{x,y}  \mid i=1,\dots,5\,, x \in \Z, \, 0 < y < N \right\}\,.\]
	These edges together with their weights are given by the following list:
	\begin{enumerate} [(1)]
		\item  $e^1_{x,y}$ is a vertical edge from $x^\circ_{x,y}$ to $x^\circ_{x,y-1}$ with weight $w(e^1_{x,y}) = 1$. 
		\item  $e^2_{x,y}$ is a diagonal edge from $x^\circ_{x,y}$ to $x^\circ_{x+1,y-1}$ with weight $w(e^2_{x,y}) = y^{-a_x}$. 
		\item $e^3_{x,y}$ is a horizontal edge from $x^\circ_{x,y}$ to $x^\bullet_{x+1,y}$ with weight $w(e^3_{x,y}) = t\cdot y^{-a_x}$. 
		\item $e^4_{x,y}$ is a diagonal edge from $x^\bullet_{x,y}$ to $x^\circ_{x+1,y-1}$ with weight $w(e^4_{x,y}) = y^{-a_x}$. 
		\item $e^5_{x,y}$ is a horizontal edge from $x^\bullet_{x,y}$ to $x^\bullet_{x+1,y}$ with weight $w(e^5_{x,y}) = t\cdot  y^{-a_x}$. 
	\end{enumerate}
	
	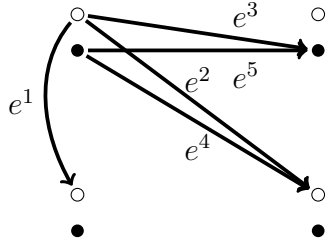
\begin{figure}[H]
		\centering
		\begin{tikzpicture}[scale=0.8]
\def\rad{3pt}

\coordinate (A1) at (1,4.3);
\coordinate (A2) at (1,3.7);

\coordinate (B1) at (1,1.3);
\coordinate (B2) at (1,0.7);

\coordinate (C1) at (5,4.3);
\coordinate (C2) at (5,3.7);

\coordinate (D1) at (5,1.3);
\coordinate (D2) at (5,0.7);

\draw [line width=0.5mm,->,shorten <= \rad+1pt, shorten >= \rad+1pt]  (A1) to [out=230,in=120] node[pos=0.5,left] {$e^1$} (B1);

\draw [line width=0.5mm,->,shorten <= \rad+1pt, shorten >= \rad+1pt]  (A1) to node[pos=0.5,above] {$e^2$} (D1);

\draw [line width=0.5mm,->,shorten <= \rad+1pt, shorten >= \rad+1pt]  (A1) to node[pos=0.7,above] {$e^3$} (C2);

\draw [line width=0.5mm,->,shorten <= \rad+1pt, shorten >= \rad+1pt]  (A2) to node[pos=0.5,below] {$e^4$} (D1);

\draw [line width=0.5mm,->,shorten <= \rad+1pt, shorten >= \rad+1pt]  (A2) to node[pos=0.7,below] {$e^5$} (C2);

\draw (A1) circle (\rad);
\fill[black] (A2) circle (\rad);

\draw (B1) circle (\rad);
\fill[black] (B2) circle (\rad);

\draw (C1) circle (\rad);
\fill[black] (C2) circle (\rad);

\draw (D1) circle (\rad);
\fill[black] (D2) circle (\rad);

\end{tikzpicture}\vspace*{-7mm}
		\caption{Overview of the five  different types of edges.}\label{fig:tlattice1}
	\end{figure}
	\vspace*{3mm}
\end{enumerate}
\end{definition}

\begin{ex}

As an example consider the following path $P$ in $G_5$ from the vertex $v^\circ_{-4,4}$ to $v^\circ_{4,0}$.
	
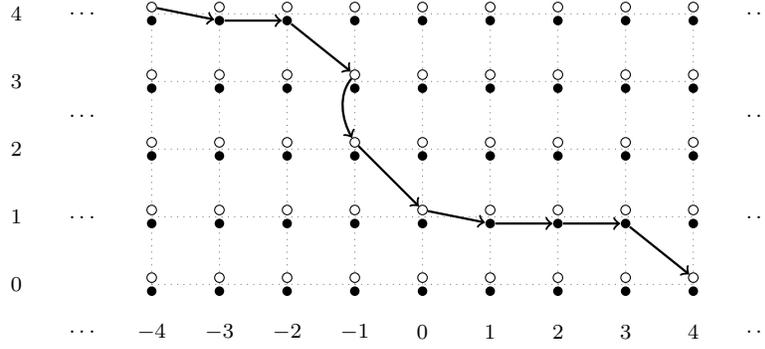
\begin{figure}[H]
		\centering
		\begin{tikzpicture}[scale=0.9]
\newcommand{\dpt}[2]{	
	\fill[white] (#1,#2+0.1) circle (2pt);
	\draw (#1,#2+0.1) circle (2pt);
	\fill[black] (#1,#2-0.1) circle (2pt);
}
\draw[step=1, thin, gray,dotted] (-4,0) grid (4,4); 

\draw [line width=0.3mm,->,shorten <= 2pt, shorten >= 2pt]  (-4,4.1) to (-3,3.9);
\draw [line width=0.3mm,->,shorten <= 2pt, shorten >= 2pt]  (-3,3.9) to (-2,3.9);
\draw [line width=0.3mm,->,shorten <= 2pt, shorten >= 2pt]  (-2,3.9) to (-1,3.1);
\draw [line width=0.3mm,->,shorten <= 2pt, shorten >= 2pt]  (-1,3.1) to [out=230,in=120] (-1,2.1);
\draw [line width=0.3mm,->,shorten <= 2pt, shorten >= 2pt]   (-1,2.1) to  (0,1.1);
\draw [line width=0.3mm,->,shorten <= 2pt, shorten >= 2pt]   (0,1.1) to   (1,0.9);
\draw [line width=0.3mm,->,shorten <= 2pt, shorten >= 2pt]   (1,0.9) to   (2,0.9);
\draw [line width=0.3mm,->,shorten <= 2pt, shorten >= 2pt]   (2,0.9) to   (3,0.9);
\draw [line width=0.3mm,->,shorten <= 2pt, shorten >= 2pt]   (3,0.9) to   (4,0.1);
\node (c) at (-5,-0.7) {\scriptsize $\dots$};
\node (c) at (-4,-0.7) {\scriptsize $-4$};
\node (c) at (-3,-0.7) {\scriptsize $-3$};
\node (c) at (-2,-0.7) {\scriptsize $-2$};
\node (c) at (-1,-0.7) {\scriptsize $-1$};
\node (c) at (0,-0.7) {\scriptsize $0$};
\node (c) at (1,-0.7) {\scriptsize $1$};
\node (c) at (2,-0.7) {\scriptsize $2$};
\node (c) at (3,-0.7) {\scriptsize $3$};
\node (c) at (4,-0.7) {\scriptsize $4$};
\node (c) at (5,-0.7) {\scriptsize $\dots$};
\node (c) at (5,1) {\scriptsize $\dots$};
\node (c) at (5,2.5) {\scriptsize $\dots$};
\node (c) at (5,4) {\scriptsize $\dots$};
\node (c) at (-5,1) {\scriptsize $\dots$};
\node (c) at (-5,2.5) {\scriptsize $\dots$};
\node (c) at (-5,4) {\scriptsize $\dots$};
\node (c) at (-6,1) {\scriptsize $1$};
\node (c) at (-6,0) {\scriptsize $0$};
\node (c) at (-6,2) {\scriptsize $2$};

\node (c) at (-6,3) {\scriptsize $3$};
\node (c) at (-6,4) {\scriptsize $4$};

\foreach \x in {-4,...,4} {
\dpt{\x}{0}	
\dpt{\x}{1}	
\dpt{\x}{2}
\dpt{\x}{3}	
\dpt{\x}{4}		
	}

\end{tikzpicture}\vspace*{-7mm}
	\end{figure}
The edges of this paths are given by $P=e^3_{-4,4}\,e^5_{-3,4}\,e^4_{-2,4}\,e^1_{-1,3}\,e^2_{-1,2}\,e^3_{0,1}\,e^5_{1,1}\,e^5_{2,1}\,e^4_{3,1}$ and we will write also for short $P=e^3 e^5 e^4 e^1 e^2 e^3 \{e^5\}^2 e^4$. Here $\{e^5\}^2$ denotes $e^5 e^5$ and we will use this notation in the following. The weight of this path is given by
\[w(P) = t 4^{-a_{-4}} \cdot t 4^{-a_{-3}} \cdot 4^{-a_{-2}} \cdot 1 \cdot  2^{-a_{-1}} \cdot t 1^{-a_{0}} \cdot t 1^{-a_{1}}  \cdot t 1^{-a_{2}} \cdot 1^{-a_{3}} =\frac{t^5}{ 2^{a_{-1}} 4^{a_{-2}} 4^{a_{-3}} 4^{a_{-4}}     } \,, \]
which is one summand in the definition of $\zeta_5^t(a_{3},a_{2},a_{1},a_{0},a_{-1},a_{-2},a_{-3},a_{-4})$.
\end{ex}
The example illustrates the idea behind the following Lemma.
\begin{lemma}\label{lem:lem0}
For integers $i\leq j$ the sum of all paths weights of paths between the vertex $A=v_{i,N-1}^\circ$ and $B=v_{j+1,0}^\circ$ is given by
\[w(A,B) = \zeta^t_N(a_j, a_{j-1},\dots,a_i) \,.\]
\end{lemma}
\begin{proof}
We will proof this by induction on $N$. The case $N=1$ is clear, since both expressions are zero. 
For the case $N\geq 1$ we use a recursive expression for the interpolated multiple zeta values $\zeta^t$. Given numbers $0 < m_1 \leq \dots \leq m_r$ we recall that $e(m_1,\dots,m_r)$ is the numbers of $=$ between $m_1,m_2,\dots,m_r$. With this we have the following recursive expression
\begin{align*}
\zeta^t_N(a_j,a_{j-1},\dots,a_i) &= \sum_{0 \leq m_1 \leq \dots \leq m_r <N} \frac{t^{e(m_1,\dots,m_r)}}{m_1^{a_j} \dots m_r^{a_i}} \\
&= \sum_{g=0}^{r-1} t^g \sum_{0 < m_1 \leq ... \leq m_{r-g-1} < m_{r-g} = \dots = m_r <N} \frac{t^{e(m_1,\dots,m_{r-g-1})}}{m_1^{a_j} \dots m_r^{a_i}}\\
&=  \sum_{g=0}^{r-1} t^g \sum_{m_r=1}^{N-1} \frac{1}{m_r^{a_i+\dots+a_{i+g}}} \zeta^t_{m_r}(a_j,a_{j-1},\dots,a_{i+g+1}) \,.
\end{align*}
For $w(A,B)$ we have a similar recursive expression. Every path from $A$ to $B$ either starts with $\{e^1\}^a e^2$ for $a=0,\dots,N-2$ or it starts with $\{e^1\}^a e^3  \{e^5\}^{g-1} e^4$ for $a=0,\dots,N-2$ and $g=1,\dots,r-1$. Therefore setting $m_r=N-1-a$ we get the same recursive formula as before
\begin{align*}
w(A,B)  &= \sum_{g=0}^{r-1} t^g \sum_{m_r=1}^{N-1} \frac{1}{m_r^{a_i+\dots+a_{i+g-1}}} w(v_{i+g+1,m_r-1}^\circ,B) \,.
\end{align*}
By the induction hypothesis it is $ w(v_{i+g+1,m_r-1}^\circ,B)=\zeta^t_{m_r}(a_j,a_{j-1},\dots,a_{i+g+1})$ from which the statement follows.
\end{proof}
		
In the following we will explain the connection of truncated interpolated Schur multiple zeta values and the $t$-Lattice Graph. For this we need some notation, which will be used in the two main Lemmas afterward. 

\begin{definition}\label{def:F}\begin{enumerate}[i)]
\item Suppose we have a partition $\lambda$ and a tuple $b=(b_1,\dots,b_{\lambda_1})$  with $b_1 \geq b_2 \geq \dots \geq b_{\lambda_1} \geq 0$ and $\lambda'_j \geq b_j$ for all $j=1,\dots,{\lambda_1}$. With this we define the Young Tableaux $F(\lambda,b):=(\lambda , f_{i,j})$ with entries $f_{i,j} \in  \{  0,1 \}$ given by \[ f_{i,j} = \begin{cases} 1 \,, & \text{ if } i > b_i \\ 0 \,, & \text{ else}
\end{cases}\,.\]
As an example we have with  $\lambda=(4,2,2,1)$ ($\lambda'=(4,3,1,1)$) and $b=(2,1,1,0)$
\[ F_1 = F(\lambda,b) = F((4,3,1,1)',(2,1,1,0)) = {\ytableausetup{centertableaux, boxsize=1.3em}	\begin{ytableau}
	0 & 0 &  0 & 1 \\
	0 &  1\\
	1 & 1 \\
	1
	\end{ytableau}}\,. \]
\item We call a tableaux $F(\lambda,b)$ 1-\emph{ordered}, when there are no $(i,j), (i+1,j+1) \in D(\lambda)$, with $f_{i,j} = f_{i+1,j+1} = 1$. In analogy to the horizontal and vertical equalities we define
\begin{align*}
v_1(F(\lambda,b,M)) &=  \#\left\{ (i,j) \in D(\lambda) \mid f_{i,j} = f_{i+1,j} = 1\right\} \,, \\
h_1(F(\lambda,b,M)) &=  \#\left\{ (i,j) \in D(\lambda) \mid f_{i,j} = f_{i,j+1} = 1\right\}  \,,
\end{align*}
which counts the number of vertical and horizontal equalities of cells containing a $1$. For the above example it is  $v_1(F_1)=2$ and $h_1(F_1)=1$.
\end{enumerate}
		
\end{definition}

\begin{lemma} \label{lem:lem1} For $M \geq 1$, a partition $\lambda=(\lambda_1,\dots,\lambda_r)$ and a tuple $b=(b_1,\dots,b_{\lambda_1})$  with $b_1 \geq b_2 \geq \dots \geq b_{\lambda_1} \geq 0$ and $\lambda'_j \geq b_j$ for all $j=1,\dots,{\lambda_1}$, define the vertex sets $\mathcal{A} = \{A_1,\dots,A_{\lambda_1} \}$ and  $\mathcal{B} = \{B_1,\dots,B_{\lambda_1}\}$ by $A_j = v^\circ_{j-\lambda'_j,M}$ and $B_j = v^\circ_{j-b_j,M-1}$. Setting $F=F(\lambda,b)$ and $a = \prod_{(i,j) \in D(\lambda)} f_{i,j} \cdot a_{j-i}$ we have
\begin{equation} \label{eq:lem1}
\sum_{\mathcal{P} \,:\,\mathcal{A} \rightarrow \mathcal{B}} \sign\mathcal{P} \cdot w(\mathcal{P}) = \begin{cases}
M^{-a} \cdot t^{v_1(F)} \cdot  (1-t)^{h_1(F)}\,, & F \text{ is 1-ordered} \\
0 \,, & \text{else}
\end{cases} \,.
\end{equation}

\end{lemma}
\begin{proof}
We will prove this by induction on $\lambda_1$. In the case $\lambda_1=1$ the Young Tableaux $F=F(\lambda,b)$ is just one column and therefore it is always 1-ordered and $h_1(F)=0$. In this case there is just one path $P$ from $A_1$ to $B_1$ given by $P=e_1$ if $b_1=\lambda_1'$, $P=e_2$ if $b_1=\lambda_1'-1$ and $P=e_3 \{e_5\}^{\lambda_1'-b_1-2} e_4$ otherwise. The weights are given by $w(P)=1$, $w(P)=M^{a_{1-\lambda'_1}}$ and $w(P)=t^{\lambda_1'-b_1-1} M^{a_{1-\lambda'_1} \cdot a_{1-(\lambda'_1-1)}  \dots  a_{1-(b_1+1)}}$ respectively. These are are the corresponding $\lambda_1=1$ cases of the right-hand side of \eqref{eq:lem1}. 

Now suppose \eqref{eq:lem1} is true for a fixed $\lambda_1 = \Lambda$, i.e.  the right-hand side gives the sum of the weights of all possible vertex-disjoint path systems in the following diagram.
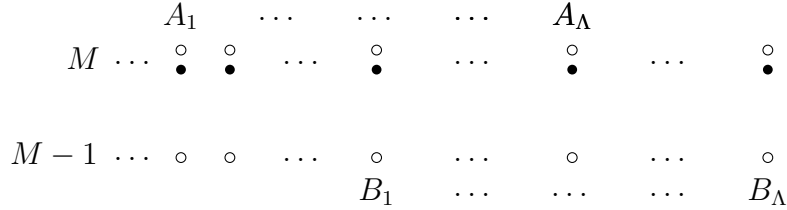
\begin{figure}[!htb]
    \centering	
    \begin{tikzpicture}[scale=1.3]

\def\ofs{0.1}
\coordinate (A1) at (0,2+\ofs);
\coordinate (A12) at (0,2-\ofs);

\coordinate (AA1) at (0.5,2+\ofs);
\coordinate (AA12) at (0.5,2-\ofs);

\coordinate (A2) at (2,2+\ofs);
\coordinate (A22) at (2,2-\ofs);

\coordinate (A3) at (4,2+\ofs);
\coordinate (A32) at (4,2-\ofs);

\coordinate (A4) at (6,2+\ofs);
\coordinate (A42) at (6,2-\ofs);

\coordinate (B1) at (0,1);
\coordinate (BB1) at (0.5,1);
\coordinate (B2) at (2,1);
\coordinate (B3) at (4,1);
\coordinate (B4) at (6,1);
\def\rad{0.05}

%
%
%
%
%
%
%
%
%
%
%
%

\draw [black] ($(-1,1.80)$) node[above] {$M$};
\draw [black] ($(-0.5,1.85)$) node[above] {${\scriptsize \dots}$};

\draw [black] ($(-1.3,0.80)$) node[above] {$M-1$};
\draw [black] ($(-0.5,0.86)$) node[above] {${\footnotesize \dots}$};

\draw [black] ($(A1)+(0,0.1)$) node[above] {$A_1$};
\draw [black] ($(1,2.1+\ofs)+(0,0.1)$) node[above] {$\dots$};
\draw [black] ($(2,2.1+\ofs)+(0,0.1)$) node[above] {$\dots$};
\draw [black] ($(3,2.1+\ofs)+(0,0.1)$) node[above] {$\dots$};
\draw [black] ($(3,2.1+\ofs)+(0,0.1)$) node[above] {$\dots$};

\draw [black] ($(1.25,1.85)$) node[above] {$\dots$};
\draw [black] ($(3,1.85)$) node[above] {$\dots$};
\draw [black] ($(5,1.85)$) node[above] {$\dots$};

\draw [black] ($(1.25,0.85)$) node[above] {$\dots$};
\draw [black] ($(3,0.85)$) node[above] {$\dots$};
\draw [black] ($(5,0.85)$) node[above] {$\dots$};

\draw [black] ($(A3)+(0,0.1)$) node[above] {$A_{\Lambda}$};
\draw [black] ($(A3)+(0,0.1)$) node[above] {$A_{\Lambda}$};

\draw [black] ($(B2)-(0,0.1)$) node[below] {$B_{1}$};
\draw [black] ($(B4)-(0,0.1)$) node[below] {$B_{\Lambda}$};

\draw [black] ($(3,0.5)$) node[above] {$\dots$};
\draw [black] ($(4,0.5)$) node[above] {$\dots$};
\draw [black] ($(5,0.5)$) node[above] {$\dots$};

\fill[white] (A1) circle (\rad);
\draw (A1) circle (\rad);
\fill[black] (A12) circle (\rad);

\fill[white] (AA1) circle (\rad);
\draw (AA1) circle (\rad);
\fill[black] (AA12) circle (\rad);

\fill[white] (A2) circle (\rad);
\draw (A2) circle (\rad);
\fill[black] (A22) circle (\rad);

\fill[white] (A3) circle (\rad);
\draw (A3) circle (\rad);
\fill[black] (A32) circle (\rad);

\fill[white] (A4) circle (\rad);
\draw (A4) circle (\rad);
\fill[black] (A42) circle (\rad);

\fill[white] (B1) circle (\rad);
\draw (B1) circle (\rad);

\fill[white] (BB1) circle (\rad);
\draw (BB1) circle (\rad);

\fill[white] (B2) circle (\rad);
\draw (B2) circle (\rad);

\fill[white] (B3) circle (\rad);
\draw (B3) circle (\rad);

\fill[white] (B4) circle (\rad);
\draw (B4) circle (\rad);

\end{tikzpicture}
\end{figure}

We now want to add additional vertices $A_{\Lambda+1}$ and $B_{\Lambda+1}$ on the right of this diagram and define $\mathcal{A}^{(2)} = \{A_1,\dots,A_{\Lambda}, A_{\Lambda+1} \}$ and  $\mathcal{B}^{(2)}  = \{B_1,\dots,B_{\Lambda}, B_{\Lambda+1}\}$. Again by $\lambda=(\lambda_1,\dots,\lambda_l)$ with $\lambda_1 = \Lambda+1$ and $b=(b_1,\dots,b_{\Lambda+1})$ we denote the corresponding partition and tuple with  $A_j = v^\circ_{j-\lambda'_j,M}$ and $B_j = v^\circ_{j-b_j,M-1}$.

{\bf First case:} $B_{\Lambda}$ is on the left of $A_{\Lambda+1}$: In this case every path system from $\mathcal{A}^{(2)} $ to $\mathcal{B}^{(2)} $ consists of the ones from $\mathcal{A}$ to $\mathcal{B}$ plus a path from $A_{\Lambda+1}$ to $B_{\Lambda+1}$, since there can not be a path to the left from $A_{\lambda_1+1}$. In other words we have
\begin{equation}\label{eq:eqind1}
\sum_{\mathcal{P} \,:\,\mathcal{A}^{(2)} \rightarrow \mathcal{B}^{(2)}} \sign\mathcal{P} \cdot w(\mathcal{P}) = \sum_{\mathcal{P} \,:\,\mathcal{A} \rightarrow \mathcal{B}} \sign\mathcal{P} \cdot w(\mathcal{P}) \cdot \sum_{P: A_{\lambda_1+1} \rightarrow B_{\lambda_1+1} } w(P) \,.
\end{equation}

If $B_{\Lambda}$ is on the left of $A_{\Lambda+1}$, then it is $b_{\Lambda} \geq \lambda'_{\Lambda+1}$, which means that $F(\lambda,b)$ has the shape
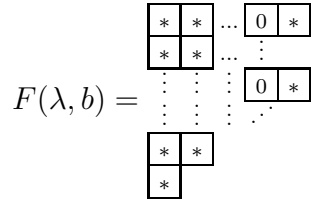
\[ F(\lambda,b) = 	\ytableausetup{mathmode, boxsize=1em}
	\begin{ytableau}
	 \scriptstyle \ast &  \scriptstyle \ast  & \none[ \scriptstyle  \dots]
	& \scriptstyle \scriptstyle 0 &  \scriptstyle \ast \\
	\scriptstyle \ast & \scriptstyle \ast  & \none[ \scriptstyle \dots]
	& \none[ \scriptstyle  \vdots] \\
	\none[\scriptstyle  \vdots] & \none[ \scriptstyle  \vdots]
	& \none[\scriptstyle  \vdots]  & \scriptstyle 0 &  \scriptstyle \ast  \\
	\none[\scriptstyle  \vdots] & \none[ \scriptstyle  \vdots]
	& \none[\scriptstyle  \vdots]  & \none[ \scriptstyle  \iddots]\\
	\scriptstyle  \ast & \scriptstyle \ast \\
	\scriptstyle \ast
	\end{ytableau}
	\]
In particular $h_1(F(\lambda,b))$ does not change by removing the last (new) column. By the induction hypothesis and \eqref{eq:eqind1} we therefore obtain \eqref{eq:lem1} for the first case. 

{\bf Second case:}  $B_{\Lambda}$ is below $A_{\Lambda+1}$: In this case there can be a path from $A_{\Lambda+1}$ to $B_{\Lambda}$ given by an vertical edge  $e_1$. Then there is also a path $P_i : A_i \rightarrow B_{\Lambda+1}$, for some $i=1,\dots,\Lambda$, passing below $A_{\Lambda+1}$ on the black vertex $\bullet$. For each  vertex-disjoint path system $\mathcal{P} = (P_1,\dots,P_{\Lambda+1})$ with this configuration there is another vertex-disjoint path system $\mathcal{P}'=(P'_1,\dots,P'_{\Lambda+1})$, which equals $\mathcal{P}$ except for $P_i$ and $P_{\Lambda+1}$, which differ in the following way
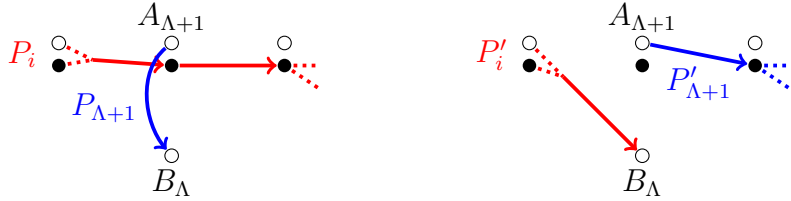
\begin{figure}[!htb]
    \centering
	\begin{tikzpicture}[scale=1.5]

\def\ofs{0.1}
\coordinate (A) at (0,2-\ofs);
\coordinate (B) at (0,2+\ofs);
\coordinate (C) at (3*\ofs,2-0.05);

\coordinate (D) at (1,2-\ofs);
\coordinate (E) at (1,2+\ofs);
\coordinate (E1) at (1-0.5,2+\ofs+0.4);
\coordinate (E2) at (1-0.2,2+\ofs+0.4);
\coordinate (E3) at (1-0.05,2+\ofs+0.4);

\coordinate (F) at (2,2-\ofs);
\coordinate (F1) at (2+0.3,2-\ofs);
\coordinate (F2) at (2+0.3,2-0.3);
\coordinate (I) at (2,2+\ofs);

\coordinate (G) at (1,1+\ofs);
\coordinate (G1) at (1.3,1);
\coordinate (G2) at (1.3,1-\ofs);
\coordinate (G3) at (0.85,1-0.08);
\coordinate (H) at (1,1-\ofs);

\def\rad{0.06}

\draw [red,line width=0.5mm,-,dotted]  (A) to (C);
\draw [red,line width=0.5mm,-,dotted]  (B) to (C);
\draw [red,line width=0.5mm,->, shorten >= 3pt-\rad]  (C) to (D);
\draw [red,line width=0.5mm,->,shorten <= 3pt+\rad, shorten >= 3pt-\rad]  (D) to (F);

\draw [red,line width=0.5mm,-,dotted]  (F) to (F1);
\draw [red,line width=0.5mm,-,dotted]  (F) to (F2);

\draw [blue, line width=0.5mm,->,shorten <= 3pt+\rad, shorten >= 3pt+\rad]  (E) to [out=220,in=130] (G);

\draw [red] ($(B)-(0.1,\ofs)$) node[left] {$P_i$};
\draw [blue] (0.4,1.3) node[above] {$P_{\Lambda+1}$};
\draw (E) node[above] {$A_{\Lambda+1}$};
\draw (1,0.65) node[above] {$B_{\Lambda}$};

\fill[white] (B) circle (\rad);
\draw (B) circle (\rad);
\fill[black] (A) circle (\rad);

\fill[white] (E) circle (\rad);
\draw (E) circle (\rad);
\fill[black] (D) circle (\rad);

\fill[white] (I) circle (\rad);
\draw (I) circle (\rad);
\fill[black] (F) circle (\rad);

\fill[white] (G) circle (\rad);
\draw (G) circle (\rad);

\end{tikzpicture}
    \qquad \qquad 
    \begin{tikzpicture}[scale=1.5]

\def\ofs{0.1}
\coordinate (A) at (0,2-\ofs);
\coordinate (B) at (0,2+\ofs);
\coordinate (C) at (3*\ofs,2-0.2);

\coordinate (D) at (1,2-\ofs);
\coordinate (E) at (1,2+\ofs);
\coordinate (E1) at (1-0.5,2+\ofs+0.4);
\coordinate (E2) at (1-0.2,2+\ofs+0.4);
\coordinate (E3) at (1-0.05,2+\ofs+0.4);

\coordinate (F) at (2,2-\ofs);
\coordinate (F1) at (2+0.3,2-\ofs);
\coordinate (F2) at (2+0.3,2-0.3);
\coordinate (I) at (2,2+\ofs);

\coordinate (G) at (1,1+\ofs);
\coordinate (G1) at (1.3,1);
\coordinate (G2) at (1.3,1-\ofs);
\coordinate (G3) at (0.85,1-0.08);
\coordinate (H) at (1,1-\ofs);

\def\rad{0.06}

\draw [red,line width=0.5mm,-,dotted]  (A) to (C);
\draw [red,line width=0.5mm,-,dotted]  (B) to (C);
\draw [red,line width=0.5mm,->, shorten >= 3pt-\rad]  (C) to (G);

\draw [blue,line width=0.5mm,-,dotted]  (F) to (F1);
\draw [blue,line width=0.5mm,-,dotted]  (F) to (F2);

\draw [blue, line width=0.5mm,->,shorten <= 3pt+\rad, shorten >= 3pt+\rad]  (E) to (F);

\draw [red] ($(B)-(0.1,\ofs)$) node[left] {$P^\prime_{i}$};
\draw [blue] (1.5,1.5) node[above] {$P^\prime_{\Lambda+1}$};

\draw (E) node[above] {$A_{\Lambda+1}$};
\draw (1,0.65) node[above] {$B_{\Lambda}$};

\fill[white] (B) circle (\rad);
\draw (B) circle (\rad);
\fill[black] (A) circle (\rad);

\fill[white] (E) circle (\rad);
\draw (E) circle (\rad);
\fill[black] (D) circle (\rad);

\fill[white] (I) circle (\rad);
\draw (I) circle (\rad);
\fill[black] (F) circle (\rad);

\fill[white] (G) circle (\rad);
\draw (G) circle (\rad);

\end{tikzpicture}
\end{figure}

or which differ by (in the case $B_{\Lambda+1}$ is next to $B_{\Lambda}$)

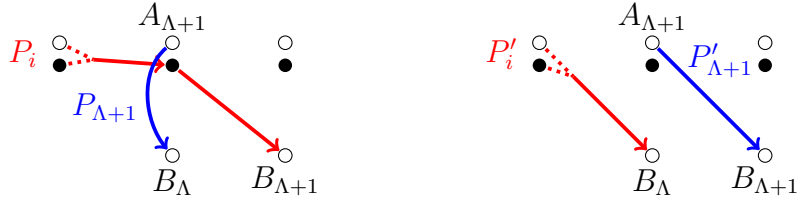
\begin{figure}[!htb]
    \centering
	\begin{tikzpicture}[scale=1.5]

\def\ofs{0.1}
\coordinate (A) at (0,2-\ofs);
\coordinate (B) at (0,2+\ofs);
\coordinate (C) at (3*\ofs,2-0.05);

\coordinate (D) at (1,2-\ofs);
\coordinate (E) at (1,2+\ofs);
\coordinate (E1) at (1-0.5,2+\ofs+0.4);
\coordinate (E2) at (1-0.2,2+\ofs+0.4);
\coordinate (E3) at (1-0.05,2+\ofs+0.4);

\coordinate (F) at (2,2-\ofs);
\coordinate (F1) at (2+0.3,2-\ofs);
\coordinate (F2) at (2+0.3,2-0.3);
\coordinate (I) at (2,2+\ofs);

\coordinate (G) at (1,1+\ofs);
\coordinate (G2) at (2,1+\ofs);

\coordinate (H) at (1,1-\ofs);

\def\rad{0.06}

\draw [red,line width=0.5mm,-,dotted]  (A) to (C);
\draw [red,line width=0.5mm,-,dotted]  (B) to (C);
\draw [red,line width=0.5mm,->, shorten >= 3pt-\rad]  (C) to (D);
\draw [red,line width=0.5mm,->,shorten <= 3pt+\rad, shorten >= 3pt-\rad]  (D) to (G2);

\draw [blue, line width=0.5mm,->,shorten <= 3pt+\rad, shorten >= 3pt+\rad]  (E) to [out=220,in=130] (G);

\draw [red] ($(B)-(0.1,\ofs)$) node[left] {$P_i$};
\draw [blue] (0.4,1.3) node[above] {$P_{\Lambda+1}$};
\draw (E) node[above] {$A_{\Lambda+1}$};
\draw (1,0.65) node[above] {$B_{\Lambda}$};

\draw (2,0.65) node[above] {$B_{\Lambda+1}$};

\fill[white] (B) circle (\rad);
\draw (B) circle (\rad);
\fill[black] (A) circle (\rad);

\fill[white] (E) circle (\rad);
\draw (E) circle (\rad);
\fill[black] (D) circle (\rad);

\fill[white] (I) circle (\rad);
\draw (I) circle (\rad);
\fill[black] (F) circle (\rad);

\fill[white] (G) circle (\rad);
\draw (G) circle (\rad);
\fill[white] (G2) circle (\rad);
\draw (G2) circle (\rad);
\end{tikzpicture}
    \qquad \qquad 
    \begin{tikzpicture}[scale=1.5]

\def\ofs{0.1}
\coordinate (A) at (0,2-\ofs);
\coordinate (B) at (0,2+\ofs);
\coordinate (C) at (3*\ofs,2-0.2);

\coordinate (D) at (1,2-\ofs);
\coordinate (E) at (1,2+\ofs);
\coordinate (E1) at (1-0.5,2+\ofs+0.4);
\coordinate (E2) at (1-0.2,2+\ofs+0.4);
\coordinate (E3) at (1-0.05,2+\ofs+0.4);

\coordinate (F) at (2,2-\ofs);
\coordinate (F1) at (2+0.3,2-\ofs);
\coordinate (F2) at (2+0.3,2-0.3);
\coordinate (I) at (2,2+\ofs);

\coordinate (G) at (1,1+\ofs);
\coordinate (G1) at (1.3,1);
\coordinate (G2) at (1.3,1-\ofs);
\coordinate (G3) at (0.85,1-0.08);
\coordinate (H) at (1,1-\ofs);
\coordinate (G2) at (2,1+\ofs);
\def\rad{0.06}

\draw [red,line width=0.5mm,-,dotted]  (A) to (C);
\draw [red,line width=0.5mm,-,dotted]  (B) to (C);
\draw [red,line width=0.5mm,->, shorten >= 3pt-\rad]  (C) to (G);

\draw [blue, line width=0.5mm,->,shorten <= 3pt+\rad, shorten >= 3pt+\rad]  (E) to (G2);

\draw [red] ($(B)-(0.1,\ofs)$) node[left] {$P^\prime_{i}$};
\draw [blue] (1.6,1.7) node[above] {$P^\prime_{\Lambda+1}$};

\draw (E) node[above] {$A_{\Lambda+1}$};
\draw (1,0.65) node[above] {$B_{\Lambda}$};
\draw (2,0.65) node[above] {$B_{\Lambda+1}$};

\fill[white] (B) circle (\rad);
\draw (B) circle (\rad);
\fill[black] (A) circle (\rad);

\fill[white] (E) circle (\rad);
\draw (E) circle (\rad);
\fill[black] (D) circle (\rad);

\fill[white] (I) circle (\rad);
\draw (I) circle (\rad);
\fill[black] (F) circle (\rad);

\fill[white] (G) circle (\rad);
\draw (G) circle (\rad);
\fill[white] (G2) circle (\rad);
\draw (G2) circle (\rad);
\end{tikzpicture}
\end{figure}
In both cases these path systems always come in pairs and since it is $\sign(\mathcal{P}) w(\mathcal{P}) = -t \sign(\mathcal{P^\prime}) w(\mathcal{P}^\prime)$ their sum gives 
 \begin{equation} \label{eq:1mt}
 (1-t) \sign(\mathcal{P^\prime}) w(\mathcal{P}^\prime) \,.\
 \end{equation}
 
 When $B_{\Lambda}$ is below $A_{\Lambda+1}$ it is $b_{\Lambda} = \lambda_{\Lambda+1}-1$ and thus $F(\lambda,b)$ has the shape
 \[ F(\lambda,b) = 	\ytableausetup{mathmode, boxsize=1em}
 	\begin{ytableau}
 	 \scriptstyle \ast &  \scriptstyle \ast  & \none[ \scriptstyle  \dots]
 	& \scriptstyle \scriptstyle 0 &  \scriptstyle \ast \\
 	\scriptstyle \ast & \scriptstyle \ast  & \none[ \scriptstyle \dots]
 	& \none[ \scriptstyle  \vdots] & \none[ \scriptstyle  \vdots] \\ 	
 	 	\none[\scriptstyle  \vdots] & \none[ \scriptstyle  \vdots]
 	 	& \none[\scriptstyle  \vdots]  & \scriptstyle 0&  \scriptstyle \ast  \\
 	\none[\scriptstyle  \vdots] & \none[ \scriptstyle  \vdots]
 	& \none[\scriptstyle  \vdots]  & \scriptstyle 1 &  \scriptstyle 1  \\
 	\none[\scriptstyle  \vdots] & \none[ \scriptstyle  \vdots]
 	& \none[\scriptstyle  \vdots]  & \none[ \scriptstyle  \iddots]\\
 	\scriptstyle  \ast & \scriptstyle \ast \\
 	\scriptstyle \ast
 	\end{ytableau}
 	\]
 	and therefore the new last column increases the number of horizontal equalities $h_1$ exactly by one. This is resembled by the additional factor of $(1-t)$ in \eqref{eq:1mt}, i.e. the induction hypothesis implies \eqref{eq:lem1} for the second case. 
 	
 {\bf Third case:}  $B_{\Lambda}$ is on the right of $A_{\Lambda+1}$: First we notice that  $B_{\Lambda}$ can just be one step to the right of $A_{\Lambda+1}$, since otherwise there can't be two vertex disjoint paths which end in $B_{\Lambda}$ and $B_{\Lambda+1}$. Similar to the second case we have for each  vertex-disjoint path system $\mathcal{P} = (P_1,\dots,P_{\Lambda+1})$ a  vertex-disjoint path system $\mathcal{P}'=(P'_1,\dots,P'_{\Lambda+1})$, which equals $\mathcal{P}$ except for $P_i$ and $P_{\Lambda+1}$, which differ in the following way
 
 
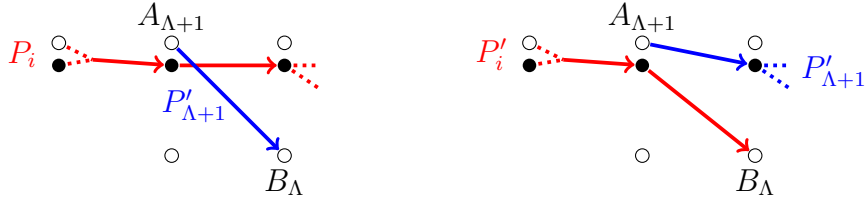
\begin{figure}[!htb]
     \centering
 	\begin{tikzpicture}[scale=1.5]

\def\ofs{0.1}
\coordinate (A) at (0,2-\ofs);
\coordinate (B) at (0,2+\ofs);
\coordinate (C) at (3*\ofs,2-0.05);

\coordinate (D) at (1,2-\ofs);
\coordinate (E) at (1,2+\ofs);
\coordinate (E1) at (1-0.5,2+\ofs+0.4);
\coordinate (E2) at (1-0.2,2+\ofs+0.4);
\coordinate (E3) at (1-0.05,2+\ofs+0.4);

\coordinate (F) at (2,2-\ofs);
\coordinate (F1) at (2+0.3,2-\ofs);
\coordinate (F2) at (2+0.3,2-0.3);
\coordinate (I) at (2,2+\ofs);

\coordinate (G) at (1,1+\ofs);
\coordinate (G1) at (1.3,1);
\coordinate (G2) at (1.3,1-\ofs);
\coordinate (G3) at (0.85,1-0.08);
\coordinate (H) at (1,1-\ofs);

\coordinate (J) at (2,1+\ofs);
\coordinate (J1) at (2.3,1);
\coordinate (J2) at (2.3,1-\ofs);
\coordinate (J3) at (1.85,1-0.08);

\coordinate (K) at (2,1-\ofs);

\def\rad{0.06}

\draw [red,line width=0.5mm,-,dotted]  (A) to (C);
\draw [red,line width=0.5mm,-,dotted]  (B) to (C);
\draw [red,line width=0.5mm,->, shorten >= 3pt-\rad]  (C) to (D);
\draw [red,line width=0.5mm,->,shorten <= 3pt+\rad, shorten >= 3pt-\rad]  (D) to (F);

\draw [red,line width=0.5mm,-,dotted]  (F) to (F1);
\draw [red,line width=0.5mm,-,dotted]  (F) to (F2);

\draw [blue, line width=0.5mm,->,shorten <= 3pt+\rad, shorten >= 3pt+\rad]  (E) to (J);

\draw [red] ($(B)-(0.1,\ofs)$) node[left] {$P_i$};

\draw (E) node[above] {$A_{\Lambda+1}$};
\draw (2,0.65) node[above] {$B_{\Lambda}$};
\draw [blue] (1.2,1.3) node[above] {$P^\prime_{\Lambda+1}$};

\fill[white] (B) circle (\rad);
\draw (B) circle (\rad);
\fill[black] (A) circle (\rad);

\fill[white] (E) circle (\rad);
\draw (E) circle (\rad);
\fill[black] (D) circle (\rad);

\fill[white] (I) circle (\rad);
\draw (I) circle (\rad);
\fill[black] (F) circle (\rad);

\fill[white] (G) circle (\rad);
\draw (G) circle (\rad);

\fill[white] (J) circle (\rad);
\draw (J) circle (\rad);

\end{tikzpicture}
     \qquad  \qquad 
     \begin{tikzpicture}[scale=1.5]

\def\ofs{0.1}
\coordinate (A) at (0,2-\ofs);
\coordinate (B) at (0,2+\ofs);
\coordinate (C) at (3*\ofs,2-0.05);

\coordinate (D) at (1,2-\ofs);
\coordinate (E) at (1,2+\ofs);
\coordinate (E1) at (1-0.5,2+\ofs+0.4);
\coordinate (E2) at (1-0.2,2+\ofs+0.4);
\coordinate (E3) at (1-0.05,2+\ofs+0.4);

\coordinate (F) at (2,2-\ofs);
\coordinate (F1) at (2+0.3,2-\ofs);
\coordinate (F2) at (2+0.3,2-0.3);
\coordinate (I) at (2,2+\ofs);

\coordinate (G) at (1,1+\ofs);
\coordinate (G1) at (1.3,1);
\coordinate (G2) at (1.3,1-\ofs);
\coordinate (G3) at (0.85,1-0.08);
\coordinate (H) at (1,1-\ofs);

\coordinate (J) at (2,1+\ofs);
\coordinate (J1) at (2.3,1);
\coordinate (J2) at (2.3,1-\ofs);
\coordinate (J3) at (1.85,1-0.08);

\coordinate (K) at (2,1-\ofs);

\def\rad{0.06}

\draw [red,line width=0.5mm,-,dotted]  (A) to (C);
\draw [red,line width=0.5mm,-,dotted]  (B) to (C);
\draw [red,line width=0.5mm,->, shorten >= 3pt-\rad]  (C) to (D);
\draw [red,line width=0.5mm,->,shorten <= 3pt+\rad, shorten >= 3pt-\rad]  (D) to (J);

\draw [blue,line width=0.5mm,-,dotted]  (F) to (F1);
\draw [blue,line width=0.5mm,-,dotted]  (F) to (F2);

\draw [blue, line width=0.5mm,->,shorten <= 3pt+\rad, shorten >= 3pt+\rad]  (E) to (F);

\draw [blue] (2.7,1.6) node[above] {$P^\prime_{\Lambda+1}$};

\draw [red] ($(B)-(0.1,\ofs)$) node[left] {$P^\prime_i$};

\draw (E) node[above] {$A_{\Lambda+1}$};
\draw (2,0.65) node[above] {$B_{\Lambda}$};

\fill[white] (B) circle (\rad);
\draw (B) circle (\rad);
\fill[black] (A) circle (\rad);

\fill[white] (E) circle (\rad);
\draw (E) circle (\rad);
\fill[black] (D) circle (\rad);

\fill[white] (I) circle (\rad);
\draw (I) circle (\rad);
\fill[black] (F) circle (\rad);

\fill[white] (G) circle (\rad);
\draw (G) circle (\rad);

\fill[white] (J) circle (\rad);
\draw (J) circle (\rad);

\end{tikzpicture}
 \end{figure}
It is $\sign(\mathcal{P}) w(\mathcal{P}) = - \sign(\mathcal{P^\prime}) w(\mathcal{P}^\prime)$	and therefore their sum gives in total $0$. Since these path systems always come in pairs, the total sum of path weights vanishes in this case. 
When $B_{\Lambda}$ is on the right of $A_{\Lambda+1}$ it is $b_{\Lambda} < \lambda_{\Lambda+1}$ and thus $F(\lambda,b)$ has the shape
 		 \[ F(\lambda,b) = 	\ytableausetup{mathmode, boxsize=1em}
 		 	\begin{ytableau}
 		 	 \scriptstyle \ast &  \scriptstyle \ast  & \none[ \scriptstyle  \dots]
 		 	& \scriptstyle \scriptstyle \ast &  \scriptstyle \ast \\
 		 	\scriptstyle \ast & \scriptstyle \ast  & \none[ \scriptstyle \dots]
 		 	& \none[ \scriptstyle  \vdots] & \none[ \scriptstyle  \vdots] \\ 	
 		 	 	\none[\scriptstyle  \vdots] & \none[ \scriptstyle  \vdots]
 		 	 	& \none[\scriptstyle  \vdots]  & \scriptstyle 1&  \scriptstyle \ast  \\
 		 	\none[\scriptstyle  \vdots] & \none[ \scriptstyle  \vdots]
 		 	& \none[\scriptstyle  \vdots]  & \scriptstyle 1 &  \scriptstyle 1  \\
 		 	\none[\scriptstyle  \vdots] & \none[ \scriptstyle  \vdots]
 		 	& \none[\scriptstyle  \vdots]  & \none[ \scriptstyle  \iddots]\\
 		 	\scriptstyle  \ast & \scriptstyle \ast \\
 		 	\scriptstyle \ast
 		 	\end{ytableau}
 		 	\]
 		 	which is clearly not 1-ordered and therefore the weight of the path system should vanish, as we showed. 
\end{proof}

We will now give the connection of path systems in the $t$-Lattice and the truncated interpolated Schur multiple zeta value. For this we fix a $N\geq 1$ and define for a partition $\lambda = (\lambda_1,\dots,\lambda_h)$ with conjugate $\lambda^\prime=(\lambda^\prime_1,\dots,\lambda^\prime_{\lambda_1})$ the sets $\mathcal{A}_\lambda := \{A_1,\dots,A_{\lambda_1}\} \subset V_N$ and $\mathcal{B}_\lambda := \{B_1,\dots,B_{\lambda_1}\} \subset V_N$ by
\begin{align}\label{def:Alam}
A_j = v^\circ_{j-\lambda_j^\prime,N-1}\quad \text{ and } \quad B_j = v^\circ_{j,0} \,.
\end{align}


With this notation we can state the following Lemma.
\begin{lemma}\label{lem:lem2}The truncated interpolated Schur multiple zeta value $\zeta^t_N(\kk) $ for the Young tableau $\kk=(\lambda, k_{i,j} )$ with $k_{i,j} = a_{j-i} $ is given by
\[  \zeta^t_N(\kk) = \sum_{\mathcal{P} \,:\,\mathcal{A}_\lambda \rightarrow \mathcal{B}_\lambda} \sign\mathcal{P} \cdot w(\mathcal{P})  \,.  \]
\end{lemma}
\begin{proof}
We will prove the statement by induction on $N$, where the $N=1$ case is clear since both sides are zero. So we assume in the following, that the statement is true for $N-1$. For a fixed vertex-disjoint path system $\mathcal{P} \,:\,\mathcal{A}_\lambda \rightarrow \mathcal{B}_\lambda$ we get a unique set $\mathcal{A}'=\{A'_1,\dots,A'_{\lambda_1}\}$ with $A'_j = v^\circ_{j-\delta'_j, N-2}$ and $ 0 \leq \delta'_j \leq \lambda'_j$ for $j=1,\dots,\lambda_1$. These are exactly the white edges $v^{\circ}_{x,y}$  on $\mathcal{P}$ with  $y=N-2$.  From this we obtain two vertex-disjoint path systems $\mathcal{P}^0 : \mathcal{A}_\lambda \rightarrow \mathcal{A}'$ and $\mathcal{P}' : \mathcal{A}'\rightarrow \mathcal{B}_\lambda$, such that $\mathcal{P}$ is the composition of these two and
\[ \sign\mathcal{P} \cdot w(\mathcal{P}) =  \sign\mathcal{P}^0 \cdot w(\mathcal{P}^0) \cdot  \sign\mathcal{P}' \cdot w(\mathcal{P}')\,.\]

Conversely we define for $\mathfrak{d}=(\delta'_1,\delta'_2 , \dots,\delta'_{\lambda_1})$, where  $\delta'_1 \geq \delta'_2 \geq \dots \geq \delta'_{\lambda_1} \geq 0$ with $ 0 \leq \delta'_j \leq \lambda'_j$ for $j=1,\dots,\lambda_1$ a set  $\mathcal{A}'(\mathfrak{d})=\{A'_1,\dots,A'_{\lambda_1}\}$ with $A'_j = v^\circ_{j-\delta'_j, N-2}$. With this we get
\[\sum_{\mathcal{P} \,:\,\mathcal{A}_\lambda \rightarrow \mathcal{B}_\lambda} \sign\mathcal{P} \cdot w(\mathcal{P})  = \sum_{\substack{\mathfrak{d}}} \left(  \sum_{\mathcal{P}^0 \,:\,\mathcal{A}_\lambda \rightarrow \mathcal{A}'(\mathfrak{d})} \sign\mathcal{P}^0  \cdot w(\mathcal{P}^0 ) \right) \left(  \sum_{\mathcal{P}' \,:\,\mathcal{A}'(\mathfrak{d}) \rightarrow \mathcal{B}_\lambda} \sign\mathcal{P}' \cdot w(\mathcal{P})' \right) \,,\]
where the first sum on the right-hand side runs over all $\mathfrak{d}=(\delta'_1,\delta'_2 , \dots,\delta'_{\lambda_1})$ with the above mentioned properties. Now we consider for such a  $\mathfrak{d}$ the number $w=\max\{ j \mid \delta'_j \neq 0\}$, which gives a partition $\delta'(\mathfrak{d}) = (\delta'_1 , \dots , \delta'_w)$ with conjugate $\delta(\mathfrak{d}) = (\delta_1,\dots, \delta_{\delta'_1})$. We will now explain that the induction hypothesis implies
\begin{equation}\label{eq:indhyp}
\zeta^t_{N-1}(\kk(\mathfrak{d}))\overset{!}{=}  \sum_{\mathcal{P}' \,:\,\mathcal{A}'(\mathfrak{d}) \rightarrow \mathcal{B}_\lambda} \sign\mathcal{P}' \cdot w(\mathcal{P})' \,,
\end{equation}
where $\kk(\mathfrak{d}) = (\delta(\mathfrak{d}), k_{i,j})$: All possible paths in a vertex-disjoint path system $\mathcal{P}' \,:\,\mathcal{A}'(\mathfrak{d}) \rightarrow \mathcal{B}_\lambda$ from $A'_{w+1},\dots,A'_{\lambda_1}$ to $B_{w+1},\dots,B_{\lambda_1}$ are given by the straight vertical paths consisting just of edges $e^1_{x,y}$. Since $w(e^1_{x,y})=1$ the only paths giving non trivial contribution to the right-hand side of \eqref{eq:indhyp} are the paths from  $A'_{1},\dots,A'_{w}$ to $B_{1},\dots,B_{w}$. These are exactly the paths in the sum of the left-hand side of the statement in the Theorem for the $N-1$ case, from which \eqref{eq:indhyp} follows. In total we therefore have 
\begin{equation}\label{eq:lem2e1}
\sum_{\mathcal{P} \,:\,\mathcal{A}_\lambda \rightarrow \mathcal{B}_\lambda} \sign\mathcal{P} \cdot w(\mathcal{P})  = \sum_{\substack{\mathfrak{d}}} \left(  \sum_{\mathcal{P}^0 \,:\,\mathcal{A}_\lambda \rightarrow \mathcal{A}'(\mathfrak{d})} \sign\mathcal{P}^0  \cdot w(\mathcal{P}^0 ) \right) \zeta^t_{N-1}(\kk(\mathfrak{d}))\,.
\end{equation}
We now want to show that this equals $\zeta_N^t(\kk)$. Recall that by definition we have 

\begin{equation} \label{eq:schurdef2}\zeta_N^t(\kk) = \sum_{\substack{\mm \in \OYT_N(\lambda)\\\mm = (\lambda, (m_{i,j}))}} t^{v(\mm)}  (1-t)^{h(\mm)} \prod_{(i,j) \in D(\lambda)} \frac{1}{m_{i,j}^{k_{i,j}}} \,. 
\end{equation}
This sum can be split into the different ways of $\mm$ being filled with $N-1$. For each $\mathfrak{d}=(\delta'_1,\delta'_2 , \dots,\delta'_{\lambda_1})$ we can associate a group of ordered Young tableaux $\mm = (\lambda, (m_{i,j})) \in \OYT_N(\lambda)$ with $m_{i,j} = N-1$ for $i > \delta'_i$ and $m_{i,j}  \leq N-1$ for $i \leq \delta'_i$. Setting \footnote{Here $\delta$ denotes the Kronecker-delta, i.e.  $\delta_{a,b}=1$ if $a=b$ and $\delta_{a,b}=0$ otherwise.} $F(\mathfrak{d}) = (\lambda, (\delta_{m_{i,j},N-1}) )$ and $a(\mathfrak{d}) = \prod_{(i,j) \in D(\lambda)} \delta_{m_{i,j},N-1} \cdot a_{j-i}$ equation \eqref{eq:schurdef2} can therefore be written as 
\begin{equation}\label{eq:lem2eq3} \zeta_N^t(\kk) =  \sum_{\substack{\mathfrak{d}}} \left( (N-1)^{-a(\mathfrak{d})} \cdot t^{v_1(F(\mathfrak{d}))} \cdot  (1-t)^{h_1(F(\mathfrak{d}))} \right) \zeta^t_{N-1}(\kk(\mathfrak{d})) \,. 
\end{equation}
Now we can use Lemma \ref{lem:lem1} for the case $b=\mathfrak{d}$ and $M=N-1$ to obtain that \eqref{eq:lem2eq3} equals \eqref{eq:lem2e1}.
\end{proof}

\begin{theorem}\label{thm:thm1harm}
For a sequence  $(a_i)_{i\in \Z}$ of arbitrary integers $a_i \in \Z$, a natural number $N\geq 1$ and a Young tableau given by $\kk=(\lambda, (k_{i,j}))$ with $k_{i,j}=a_{j-i}$, we have the following identity
	\[ \zeta^t_N(\kk) = \det(\, \zeta^t_N(a_{j-1}, a_{j-2}, \dots , a_{j-(\lambda^\prime_i+j-i)}) \,)_{1\leq i,j \leq \lambda_1}\,,\]
	where we set $\zeta^t_N(a_{j-1}, a_{j-2}, \dots , a_{j-(\lambda^\prime_i+j-i)})$ to be $\begin{cases}
1 \, \text{  if  }  \lambda^\prime_i-i +j =0  \\
0 \, \text{  if  }  \lambda^\prime_i-i +j <0 
\end{cases}$.
\end{theorem}
\begin{proof}
Define as before the sets $\mathcal{A}_\lambda := \{A_1,\dots,A_{\lambda_1}\} \subset V_N$ and $\mathcal{B}_\lambda := \{B_1,\dots,B_{\lambda_1}\} \subset V_N$ by $A_j = v^\circ_{j-\lambda_j^\prime,N-1}$ and  $B_j = v^\circ_{j,0}$.
By the Lemma of Lindstr\"om, Gessel, Viennot we have that 
\begin{equation}\label{eq:lgveq}
  \sum_{\mathcal{P} \,:\,\mathcal{A}_\lambda \rightarrow \mathcal{B}_\lambda} \sign\mathcal{P} \cdot w(\mathcal{P})= \det( w(A_i,B_j) )_{1\leq i,j \leq n}\,. 
   \end{equation}
For $\lambda'_i + j -i =0$ the vertex $A_i$ has the same $x$-coordinate as $B_j$ and therefore the only possible path from $A_i$ to $B_j$ is a straight vertical path, i.e $w(A_i,B_j)=1$. In the case $\lambda'_i + j -i  < 0$ the vertex  $B_j$ is on the left of $A_i$ and therefore $w(A_i,B_j)=0$, since there is no path from $A_i$ to $B_j$. In the case $\lambda'_i + j -i >0$ we get by Lemma \ref{lem:lem0} that $w(A_i,B_j) = \zeta^t_N(a_{j-1}, a_{j-2}, \dots , a_{j-(\lambda^\prime_i+j-i)})$. Now the statement follows by using Lemma $\ref{lem:lem2}$ on the left-hand side of \eqref{eq:lgveq}.
\end{proof}

\begin{theorem} \label{thm:thm2harm}	For $N\geq 1$ and an arbitrary sequence of integers $a_j \in \Z$, $i\in \Z$ and a partition $\lambda=(\lambda_1,\dots,\lambda_h)$ with conjugate  $\lambda^\prime=(\lambda^\prime_1,\dots,\lambda^\prime_{\lambda_1})$ the following identity holds
	\[ \det(\, \zeta_N^{1-t}(a_{1-j}, a_{2-j}, \dots , a_{(\lambda_i+j-i)-j}) \,)_{1\leq i,j \leq h}= \det(\, \zeta_N^t(a_{j-1}, a_{j-2}, \dots , a_{j-(\lambda^\prime_i+j-i)}) \,)_{1\leq i,j \leq \lambda_1}\,,\]
	where we use the same definition of  $\zeta_N^t(a_m,\dots,a_n)$ for $n<m$ as before. 
\end{theorem}
\begin{proof}
Using the formula in Theorem \ref{thm:thm1harm} for the cases $t$ and $1-t$ we obtain this result together with Proposition \ref{prop:t1mt}, which stated that $t \rightarrow 1-t$ corresponds to $\lambda \rightarrow \lambda'$.
\end{proof}

Theorem \ref{thm:thm1} and \ref{thm:thm2} in the introduction are a direct consequence of Theorem \ref{thm:thm1harm} and \ref{thm:thm2harm} by taking the limit $N \rightarrow \infty$. The condition $a_j \geq 2$ is necessary for the convergence of the elements appearing in the matrices. 

\section{Generalization}
In this section we will discuss a possible generalization of our main results. For this let $A$ be a set, $\mathcal{R}$ a commutative ring and $f$ a map
\begin{align*}
f: A \times \N &\longrightarrow \mathcal{R} \\
(k,m) &\longmapsto f(k,m) \,.
\end{align*}
Define for $k_1,\dots,k_r \in A$ the element 
\begin{equation} \label{eq:generalf}
 \F^t_N(k_1,\dots,k_r) = \sum_{0 < m_1 \leq m_2 \leq \dots \leq m_r < N} t^{e(m_1,\dots,m_r)} f(k_1, m_1) \dots f(k_r, m_r)  \in \mathcal{R}[t]\,. 
 \end{equation}
One can also define the Schur version of these objects. For this we define for a Young tableau $\kk = (\lambda, (k_{i,j}))$ with $k_{i,j} \in A$
\begin{equation}\label{def:fsmzv}
 \F_N^t(\kk) = \sum_{\substack{\mm \in \OYT_N(\lambda)\\\mm = (\lambda, (m_{i,j}))}} t^{v(\mm)}  (1-t)^{h(\mm)} \prod_{(i,j) \in D(\lambda)} f(k_{i,j},m_{i,j}) \in \mathcal{R}[t]\,. 
\end{equation}
These Schur version also satisfy the statement of Proposition \ref{prop:t1mt}, i.e. $\F_N^t(\kk) = \F_N^{1-t}(\ck)$. Using the same technique as before we get the following result. 

\begin{theorem} \label{thm:final}For $N\geq 1$ and an arbitrary sequence $a_j \in A$, $i\in \Z$ and a partition $\lambda=(\lambda_1,\dots,\lambda_h)$ with conjugate  $\lambda^\prime=(\lambda^\prime_1,\dots,\lambda^\prime_{\lambda_1})$ and a Young tableau given by $\kk=(\lambda, (k_{i,j}))$ with $k_{i,j}=a_{j-i}$, the following identities hold in $\mathcal{R}[t]$
\begin{align*}
	 \F^t_N(\kk) &= \det(\, \F^t_N(a_{j-1}, a_{j-2}, \dots , a_{j-(\lambda^\prime_i+j-i)}) \,)_{1\leq i,j \leq \lambda_1}\\
	&= \det(\, \F_N^{1-t}(a_{1-j}, a_{2-j}, \dots , a_{(\lambda_i+j-i)-j}) \,)_{1\leq i,j \leq h}\,,
	\end{align*}
where we set $\F^t_N(a_m,\dots,a_n)=1$ for $n=m-1$ and $\F^t_N(a_m,\dots,a_n)=0$ when $n<m-1$. 
\end{theorem}
\begin{proof}
To get this result one just needs to modify the weights of the edges of the $t$-Lattice Graph in Definition \ref{def:tlattice} ii)  by $w(e^1_{x,y}) = 1$, $w(e^2_{x,y}) = w(e^4_{x,y}) = f(a_x,y)$ and $w(e^3_{x,y}) = w(e^5_{x,y}) = t \cdot f(a_x,y)$. With this the statements of the Lemmas \ref{lem:lem0}, \ref{lem:lem1} and \ref{lem:lem2} are still valid. Therefore the Lemma of Lindstr\"om, Gessel, Viennot can be applied, since $\mathcal{R}[t]$ is commutative, to get an analogue version of Theorem \ref{thm:thm1harm} and Theorem \ref{thm:thm2harm}, which then implies the above statement.
\end{proof}

\begin{ex}
\begin{enumerate}[i)]
\item To obtain our (truncated) interpolated Schur multiple zeta values we choose $A=\Z$, $\mathcal{R}=\Q$ and $f(k,m)=m^{-k}$. Of course we could also choose $A=\C$, which was done in \cite{NPY}. The limit $N\rightarrow \infty$ then exists in the cases $\Re(a_j) > 1$, which follows by the well-known convergence regions of multiple zeta functions. 
\item In \cite{W} Wakabayashi introduces interpolated $q$-analogues of multiple zeta values given for $k_1,\dots,k_{r-1} \geq 1$ and $k_1 \geq 2$ by\footnote{The order of summation is different in \cite{W} and the Definition 1 there looks different to ours. But it can be checked easily that the objects are, up to the order of summation, the same.}
\begin{align*}
\zeta^t_q(k_1,\dots,k_r) = \sum_{0<m_1\leq\cdots \leq m_r} \frac{t^{e(m_1,\dots,m_r)} q^{(k_1-1)m_1+\cdots +(k_r-1)m_r}}{[m_1]_q^{k_1}\cdots [m_r]_q^{k_r}} \in \Q[[q]][t]\,,
\end{align*}
where $[m]_q = \frac{1-q^m}{1-q} = 1+ q+\dots+q^{m-1}$ denotes the usual $q$-analogue of a natural number $m$. In the case $q \rightarrow 1$ these give back the interpolated multiple zeta values $\zeta^t(k_1,\dots,k_r)$. Defining their truncated version $\zeta^t_{q,N}(k_1,\dots,k_r)$ in the obvious way and choosing $A=\N$, $\mathcal{R}=\Q[[q]]$ and $f(k,m)=q^{m(k-1)} [m]_q^{-k}$ we obtain algebraic relations for these interpolated $q$-analogues after taking the limit $N \rightarrow \infty$. 
\item Similar to Section 5 in \cite{NPY}, we can take $A=\N$, $\mathcal{R} = \Z[[ x_1,x_2,\dots]]$ and $f(k,m)=x^k_{m}$. In this case $\F^t_N(k_1,\dots,k_r)$ is an interpolaion between the monomial quasi-symmetric function $\F^0_N(k_1,\dots,k_r)$ and the essential quasi-symmetric function $\F^1_N(k_1,\dots,k_r)$. 
\end{enumerate}
\end{ex}
Besides these examples the author is not aware of any interpolated object of type $\eqref{eq:generalf}$ or their Schur versions $\eqref{def:fsmzv}$.
\end{document}